\documentclass[12pt]{amsart}
\usepackage{amssymb}
\usepackage{amsmath}
\usepackage{amsthm}
\usepackage{mathtools}
\usepackage{graphicx}
\usepackage{caption}
\usepackage{subcaption}
\usepackage{enumerate}
\usepackage{hyperref}
\newtheorem{theorem}{Theorem}[section]
\newtheorem{algorithm}[theorem]{Algorithm}
\newtheorem{remark}[theorem]{Remark}
\newtheorem{definition}[theorem]{Definition}
\newtheorem{proposition}[theorem]{Proposition}
\newtheorem{lemma}[theorem]{Lemma}

\newtheorem{prob}[theorem]{Problem}

\newtheorem{corollary}[theorem]{Corollary}

\newcommand{\argmin}[1]{\underset{{#1}}{\operatorname{argmin}}}
\newcommand{\norm}[1]{\left|\left|#1\right|\right|}

\newcommand{\var}[1]{\left( {#1} \right)}
\newcommand{\Mh}{\nabla}

\newcommand{\Lop}{L}

\begin{document}
\title[Iterated Modified Lavrentiev Deconvolution]{Iterative Refinement of a Modified Lavrentiev Regularization Method for Deconvolution of the Discrete Helmholtz-Type Differential Filter}
\author{Nathaniel Mays}
\address{Verona, WI}
\email{nate.mays@gmail.com}
\author{Ming Zhong$^*$}
\address{Department of Applied Mathematics \& Statistics, Johns Hopkins University, Baltimore, MD $21218$, USA}
\curraddr{Department of Mathematics, University of Houston, Houston, TX $77204$, USA}
\email{mzhong3@central.uh.edu}
\date{\today}
\begin{abstract}
We analyze an iterated, modified Lavrentiev regularization method (Mitlar) for deconvolving the Helmholtz-type differential filter $G=(I-\delta^2\Delta)^{-1}$ used in large eddy simulation.  The method replaces the Lavrentiev operator $G+\alpha I$ by the convex combination $(1-\alpha)G+\alpha I$; each step costs one Helmholtz solve, exactly as in iterated Lavrentiev regularization.  We show that this one-line change has a precise structural meaning: the regularized operator is again a Helmholtz filter, now with the smaller radius $\sqrt{\alpha}\,\delta$, and $J$ refinement steps coincide with van Cittert approximate deconvolution of order $J$ applied to that filter.  Consequently the regularization error is governed by the product $\alpha\delta^2$ rather than by $\alpha$ alone.  We prove sharp error bounds of order $(\alpha\delta^2)^{s}$ on a Hilbert scale, $0\le s\le J+1$, with saturation exactly at $s=J+1$; show that for noise-free data the Mitlar error never exceeds the iterated Lavrentiev error, for every datum, every $\alpha$ and every $J$; give a priori parameter choices for noisy data with explicit $\delta$-dependence; and prove finite element error bounds whose constants grow only linearly in $J$.  Finally we show that the iterates decrease a natural energy functional for every $\alpha>0$, and derive a computable stopping rule that guarantees monotone decrease of a weighted error and always terminates.  Numerical examples illustrate the theory and compare Mitlar with a family of Tikhonov--Lavrentiev methods.
\end{abstract}
\maketitle
\section{Introduction}\label{sec:intro}
Large eddy simulation (LES), introduced by Smagorinsky in $1963$ \cite{Sma63}, sits between direct numerical simulation (DNS) of the Navier--Stokes equations and Reynolds-averaged models.  It is now standard in meteorology, astrophysics, and aerospace, mechanical, chemical and environmental engineering \cite{M84, GCS10, MM97, SF97, DSH01, NY98, PSB01}.  LES computes local spatial averages $\bar u$ of the flow, so only scales larger than a filter radius $\delta$ have to be resolved \cite{PSB01, GOP04, LL05, BIL06}.  Most LES filters are convolutions,
\begin{equation}\label{eq:les_filter}
  \bar{u}(x; \delta) = \int_{\Omega} G(x - x'; \delta)\,u(x')\, dx',
\end{equation}
where $\Omega$ is the spatial domain \cite{Geu97, PSB01, GOP04, G86}.  This paper concerns the \emph{differential} (Helmholtz-type) filter of Germano \cite{G86}, defined by the elliptic problem
\begin{equation*}
  -\delta^2\Delta\bar u + \bar u = u \ \text{ in } \Omega, \qquad\text{that is,}\qquad \bar u = Gu,\quad G = (I-\delta^2\Delta)^{-1}.
\end{equation*}
It is popular because it is easy to apply in finite element codes, and it is equivalent to the Pao filter used in image processing \cite{LL05, BB98}.

Many LES closure models need an approximate inverse of the filter.  Given the filtered field $\bar u$, one has to approximately recover $u$, including scales below the cutoff $\delta$ \cite{SA99, SAK01, DE06, LL05, BIL06, LR12, Geu97}.  We study this deconvolution problem in two forms.
\begin{prob}[Noise-free model problem] \label{prob:noisefreeDiff}
Let $X$ and $Y$ be Hilbert spaces, let $G: X \rightarrow Y$ be a linear filter operator, and let $\bar{u} \in \operatorname{Range}(G)$.  Find $u \in X$ such that
\begin{equation} \label{EQ:TRUEPB}
  Gu = \bar{u}.
\end{equation}%
\end{prob}
\begin{prob}[Noisy model problem] \label{prob:noisyDiff}
Let $X$, $Y$ and $G$ be as in Problem~\ref{prob:noisefreeDiff}, and let $\epsilon\in Y$ be noise with $\|\epsilon\|\le\epsilon_0$.  Given only the perturbed data $\bar u$, find an approximation of the $u \in X$ satisfying
\begin{equation} \label{eqn:noisyproblem}
  Gu = \bar{u} + \epsilon.
\end{equation}
\end{prob}
Deconvolution problems of this kind arise in parameter identification \cite{EM09, EHN96}, in image deblurring \cite{BB98}, and in the closure problem of turbulence modeling \cite{BIL06, Geu97, LL05, MS09, LT10}.  Because $G$ is compact and injective with infinite-dimensional range, both problems are ill-posed \cite{AN01, BK04, H94, N84, S07, V02, V82}.  Small perturbations of $\bar u$ can produce arbitrarily large changes in $u$, so the problems must be regularized.

Since $G$ is self-adjoint and positive, the natural regularization is Lavrentiev's method $(G+\alpha I)u_0=\bar u$ \cite{L67, LN96} and its iterated version, which wraps a defect-correction loop around it \cite{KC79, EHN96, VV86, MS09, LR12}.  These methods are generic: they use only that $G$ is self-adjoint and nonnegative.  They ignore the fact that $G^{-1}=I-\delta^2\Delta$ is a differential operator that is close to the identity on scales larger than $\delta$.  This has a measurable cost.  For every datum $u$, the noise-free error of $J$-step iterated Lavrentiev regularization is at least $(\alpha/(1+\alpha))^{J+1}\|u\|$ (Theorem~\ref{thm:domination}).  So even a very weak filter ($\delta\to0$) is not inverted accurately unless $\alpha\to 0$, and then the noise is amplified as $1/\alpha$.

In this paper we analyze a modification that costs nothing extra.  We replace $G+\alpha I$ by the convex combination $(1-\alpha)G+\alpha I$, and iterate (Algorithm~\ref{alg:mitlar}).  We call the result \emph{modified iterated Tikhonov--Lavrentiev regularization}, or Mitlar.  The method and its first error analysis were developed jointly by the present authors.  An earlier version of that joint work appears in Chapter~3 of the first author's doctoral thesis \cite{Mays11}.  The operator $((1-\alpha)G+\alpha I)^{-1}$ and its iterated version also appear in \cite{IMMS12}, as examples of deconvolution operators in time-relaxation turbulence models.  The present paper revises and corrects that analysis and extends it.  Proposition~\ref{prop:structure}, parts (b)--(d) of Theorem~\ref{thm:noisefree}, Theorem~\ref{thm:domination}, Corollaries~\ref{cor:alpha} and~\ref{cor:J}, the bound in Theorem~\ref{thm:discerror} with constants linear in $J$, the extension of Proposition~\ref{5.1:prop1} to all $\alpha>0$, and parts (b)--(c) of Theorem~\ref{5.3:them1} are new.  Our starting point is an identity that is implicit in the variational formulation of \cite{IMMS12}:
\begin{equation*}
  \big((1-\alpha)G+\alpha I\big)^{-1}G \;=\; (I-\alpha\delta^2\Delta)^{-1}.
\end{equation*}
In words, the regularized operator is again a Helmholtz filter, now with the smaller radius $\sqrt{\alpha}\,\delta$.  Moreover, $J$ Mitlar refinement steps are exactly van Cittert approximate deconvolution of order $J$ \cite{BB98, SA99, DE06} applied to that narrower filter (Proposition~\ref{prop:structure}).  This explains why the error depends on the product $\alpha\delta^2$, and it makes a sharp analysis possible.

\medskip\noindent\textbf{Contributions.}
\begin{enumerate}[(i)]
  \item \emph{Structure} (Proposition~\ref{prop:structure}).  The Mitlar map is $D_{\alpha,J}G = I-(I-G_{\sqrt{\alpha}\delta})^{J+1}$, where $G_\gamma=(I-\gamma^2\Delta)^{-1}$.
  \item \emph{Sharp noise-free error bounds} (Theorem~\ref{thm:noisefree}).  For $u$ in the Hilbert scale generated by the Dirichlet Laplacian and $0\le s\le J+1$,
  \[
    \|u-D_{\alpha,J}Gu\|\le \Big(\frac{s}{J+1}\Big)^{s}(\alpha\delta^2)^{s}\,\|(-\Delta)^{s}u\|.
  \]
  The case $s=J+1$ is asymptotically exact, and the rate saturates exactly at $s=J+1$.  The bound also decays in $J$ when $s<J+1$.
  \item \emph{Uniform comparison} (Theorem~\ref{thm:domination}).  For noise-free data, the Mitlar error is bounded by the iterated Lavrentiev error for every $u$, every $\alpha>0$ and every $J$.  This holds for any self-adjoint filter with spectrum in $[0,1]$, including the discrete filters used in computations.
  \item \emph{Noisy data} (Theorem~\ref{thmContErrorNoise} and Corollaries~\ref{cor:alpha}--\ref{cor:J}).  We give a sharp noise-amplification constant $(J+1)/\alpha$.  We also give a priori choices of $\alpha$ (for fixed $J$) and of $J$ (for fixed $\alpha$) with explicit $\delta$-dependence; when the number of iterations is used as the regularization parameter, the rates have no saturation.
  \item \emph{Finite element discretization} (Theorem~\ref{thm:discerror} and Corollary~\ref{cor:P1}).  The error bounds have constants that grow only linearly in $J$, and they are fully explicit for piecewise linear elements on convex domains.
  \item \emph{Energy descent and stopping} (Section~\ref{sec:descent}).  The iterates decrease the natural energy functional for every $\alpha>0$, not only for $\alpha\le\frac12$.  A computable stopping rule always terminates and guarantees that a weighted error decreases monotonically up to the stopping index.
\end{enumerate}

The paper is organized as follows.  Section~\ref{sec:prelim} collects notation, properties of the differential filter, and classical Tikhonov--Lavrentiev results.  Section~\ref{sec:mitlar} introduces Mitlar and contains the continuous and discrete error analysis.  Section~\ref{sec:descent} studies descent properties and the stopping rule.  Section~\ref{sec:numerics} presents numerical examples: a test of the stopping rule, a comparison with Tikhonov--Lavrentiev, iterated Tikhonov--Lavrentiev and modified Tikhonov--Lavrentiev regularization, and a convergence study.
\section{Preliminaries and notation}\label{sec:prelim}
We use standard notation for Lebesgue and Sobolev spaces and their norms.  Throughout, $\Omega\subset\mathbb{R}^d$ is a bounded polyhedral domain, and the unknown fields may be vector valued with $d$ components.  All operators below act componentwise, so the reader may equally think of scalar functions.  Let
\begin{equation*}
    X = H^1_0(\Omega)^d = \left\{ v \in L^2(\Omega)^d: \nabla v \in L^2 (\Omega)^{d\times d} \text{ and } v = 0 \text{ on } \partial \Omega\right\}.
\end{equation*}
Unless a subscript says otherwise, $\| \cdot \|$ and $(\cdot, \cdot)$ denote the $L^2(\Omega)$ norm and inner product.  Let $X^h \subset X$ be a finite element space of continuous piecewise polynomials of degree $k\ge1$ on a shape-regular mesh of size $h$.  The following approximation properties hold \cite{BS94}: for $u \in H^{k+1}(\Omega)^d\cap X$, there is $v^h\in X^h$ (for instance a Scott--Zhang or Lagrange interpolant) with
\begin{equation}\label{eqn:approxInequality}
  \| u - v^h \| \leq C h^{k+1} | u |_{H^{k+1}(\Omega)}, \qquad
  \| \nabla(u - v^h) \| \leq C h^{k} | u |_{H^{k+1}(\Omega)},
\end{equation}
where $C$ depends only on $k$ and the shape regularity of the mesh.
\subsection{The differential filter and its spectral calculus}
The differential (Helmholtz-type) filter is used in many large eddy simulation models \cite{BIL06, G86, Geu97, LL05, MK12, MS09, LR12}.
\begin{definition}[Differential filter]\label{def:filter}
For $\gamma>0$, the differential filter $G_\gamma$ maps $u\in L^2(\Omega)^d$ to the unique $\bar u = G_\gamma u \in X$ satisfying
\begin{equation}\label{eqn:differentialfilter}
-\gamma^2 \Delta \bar u + \bar u = u \quad \text{ in } \Omega,\qquad \bar u=0\ \text{ on }\partial\Omega,
\end{equation}
in the variational sense
\begin{equation} \label{eqnVarDiffFilter}
  \gamma^2 ( \nabla \bar{u}, \nabla v ) + (\bar{u}, v) = (u,v)\qquad \forall v \in X.
\end{equation}
We write $G:=G_\delta$ for the filter with the given radius $\delta$.
\end{definition}
\begin{definition} [Discrete differential filter]
\label{defn:discdiffFilterVar}
  Define $G^h:L^2(\Omega)^d \rightarrow X^h$ by letting $\bar{u}^h = G^h u$ be the unique solution in $X^h$ of
\begin{equation*}
\delta^2 \var{\Mh \bar{u}^h,\Mh v^h} + \var{\bar{u}^h,v^h} =
\var{u,v^h}\qquad \forall v^h \in X^h.
\end{equation*}
\end{definition}

Let $\Lop:=-\Delta_D$ denote the Dirichlet Laplacian, the self-adjoint operator on $L^2(\Omega)^d$ associated with the form $(\nabla u,\nabla v)$ on $X$.  Its domain is $\mathcal D(\Lop)=\{u\in X:\Delta u\in L^2(\Omega)^d\}$.  Since $\Omega$ is bounded, $\Lop$ has an $L^2$-orthonormal basis of eigenfunctions $\{\varphi_k\}$ with eigenvalues $0<\mu_1\le\mu_2\le\cdots\to\infty$.  For $s\ge0$ we use the Hilbert scale
\begin{equation*}
  \mathcal D(\Lop^s)=\Big\{u\in L^2(\Omega)^d:\ \|\Lop^s u\|^2:=\sum_k \mu_k^{2s}|\hat u_k|^2<\infty\Big\},
\end{equation*}
where $\hat u_k=(u,\varphi_k)$.
For an integer $m\ge1$, $u\in\mathcal D(\Lop^m)$ if and only if $u,\Delta u,\dots,\Delta^{m-1}u\in X$ and $\Delta^m u\in L^2(\Omega)^d$; in that case $\|\Lop^m u\|=\|\Delta^m u\|$.  The \emph{compatibility conditions} $\Delta^i u|_{\partial\Omega}=0$ for $i<m$ are part of this assumption.  For $s=\frac12$ we have $\mathcal D(\Lop^{1/2})=X$ and $\|\Lop^{1/2}u\|=\|\nabla u\|$.

By \eqref{eqnVarDiffFilter}, $G_\gamma=(I+\gamma^2\Lop)^{-1}$.  Every operator in this paper is therefore a function of $\Lop$, acting on the eigenfunction expansion by a scalar multiplier.  We record the consequences.
\begin{lemma}\label{lem:filterspectral}
The filter $G_\gamma:L^2(\Omega)^d\to L^2(\Omega)^d$ is compact, injective, self-adjoint and positive.  Its spectrum is
\begin{equation*}
  \sigma(G_\gamma)=\{0\}\cup\Big\{\tfrac{1}{1+\gamma^2\mu_k}\Big\}_{k\ge1}\subset[0,1),
\end{equation*}
and $0$ is an accumulation point of $\sigma(G_\gamma)$, so $\operatorname{Range}(G_\gamma)$ is not closed.  For any bounded Borel function $\phi$ on $[0,1]$ and any $u\in L^2(\Omega)^d$,
\begin{equation}\label{eq:functional}
  \|\phi(G_\gamma)u\|^2=\sum_k \big|\phi\big(\tfrac{1}{1+\gamma^2\mu_k}\big)\big|^2|\hat u_k|^2,\qquad \|\phi(G_\gamma)\|=\sup_{x\in\sigma(G_\gamma)}|\phi(x)|.
\end{equation}
\end{lemma}
\begin{proof}
Everything follows from $G_\gamma\varphi_k=(1+\gamma^2\mu_k)^{-1}\varphi_k$ and $\mu_k\to\infty$.
\end{proof}

The next properties of the differential filter are proved in \cite{MK12}.  The second estimate of Lemma~\ref{lemFilterConv} also follows directly from \eqref{eq:functional}.
\begin{lemma}\label{lemFilterStability}
If $u \in L^2(\Omega)^d$ and $\bar u=Gu$, then
\begin{equation*}
  \delta^2 \| \nabla \bar{u} \|^2 + \frac{1}{2} \| \bar{u} \|^2 \leq \frac{1}{2} \| u \|^2.
\end{equation*}
\end{lemma}
\begin{lemma}\label{lemFilterConv}
 If $u \in X$ and $\bar{u}=Gu$, then
 \begin{equation*}
  \frac{\delta^2}{2} \| \nabla( u - \bar{u} ) \|^2 + \| u - \bar{u} \|^2 \leq \frac{\delta^2}{2} \| \nabla u \|^2.
 \end{equation*}
 If, in addition, $u\in\mathcal D(\Lop)$, then
  \begin{equation*}
  \delta^2 \| \nabla( u - \bar{u} ) \|^2 + \frac{1}{2} \| u - \bar{u} \|^2 \leq \delta^4 \| \Delta u \|^2.
 \end{equation*}
\end{lemma}
\begin{lemma}\label{lemDiscFilterSelfAdjoint}
The operator $G^h:L^2(\Omega)^d \rightarrow X^h$ is self-adjoint and positive semi-definite on $L^2(\Omega)^d$.  Its restriction to $X^h$ is positive definite with spectrum in $(0,1)$.
\end{lemma}
\subsection{Tikhonov and Lavrentiev regularization}
Tikhonov regularization is a basic tool for ill-posed problems \cite{TA77, TA79, EHN96}.  To solve $Gu = \bar{u}$, general Tikhonov regularization computes
\begin{equation}\label{eq:gen_tik}
  u_T = \argmin{u}\{\norm{Gu - \bar{u}}^2 + \norm{\Gamma u}^2\},
\end{equation}
where $\Gamma$ is a suitably chosen linear operator, the Tikhonov operator.  The minimizer has the closed form
\begin{equation}\label{eq:gen_tik_soln}
  u_T = (G^*G + \Gamma^*\Gamma)^{-1}G^*\bar u,
\end{equation}
where $G^*$ and $\Gamma^*$ denote the adjoints of $G$ and $\Gamma$.  The most common choice is $\Gamma = \alpha I$, which gives $u_T = (G^*G + \alpha^2 I)^{-1}G^*\bar u$.  If $G$ is monotone (for the Helmholtz filter it is self-adjoint and positive semi-definite), the normal equations can be avoided.  One instead perturbs $G$ itself.  This is the method of Lavrentiev regularization \cite{L67}, also called the method of singular perturbation \cite{LN96}.
\begin{definition}[Tikhonov--Lavrentiev regularization] \label{def:diffTikhonov}
  Choose $\alpha > 0$ and solve for $u_0$ satisfying $( G + \alpha I ) u_0 = \bar{u}$.
\end{definition}
The accuracy of this method can be improved by iteration.
\begin{definition}[Iterated Tikhonov--Lavrentiev regularization]\label{def:ITL}
Choose $\alpha>0$ and the number of updates $J \geq 0$.  The iterates $u_j$, $0 \leq j \leq J$, are defined by
\begin{align*}
  ( G + \alpha I ) u_0 &= \bar{u}, \\
  ( G + \alpha I ) ( u_j - u_{j-1} ) &= \bar{u} - G u_{j-1}, \qquad 1\le j\le J.
\end{align*}
\end{definition}
Under a source condition, Tikhonov--Lavrentiev and iterated Tikhonov--Lavrentiev regularization converge to $u$ as $\alpha\to0$ for noise-free data, and as $\epsilon_0 \rightarrow 0$ with a suitable choice $\alpha=\alpha(\epsilon_0)\to0$ for noisy data \cite{EHN96, KC79, VV86}.  The following version is standard; it follows from \eqref{eq:functional}, since the noise-free error is $(\alpha(G+\alpha I)^{-1})^{J+1}u$.
\begin{theorem}[Error bound for iterated Tikhonov--Lavrentiev regularization \cite{EHN96, KC79}]\label{thm:ITL}
  Let $G$ be self-adjoint with $\sigma(G)\subset[0,1]$, and let $\alpha>0$.  Suppose $u=G^\beta w$ for some $\beta\ge J+1$ and $w\in L^2$, and $\| \epsilon \| \leq \epsilon_0$.  Then $e_J = u - u_J$ satisfies
\begin{equation*}
  \| e_J \| \leq \frac{(J+1)\epsilon_0}{\alpha} + \alpha^{J+1}\|w\|.
\end{equation*}
In particular, the choice $\alpha = C \epsilon_0^{1/(J+2)}$ gives $\| e_J \| \leq C' \epsilon_0^{(J+1)/(J+2)}$.
\end{theorem}
The error bounds for Tikhonov--Lavrentiev regularization are similar to those for Tikhonov regularization \cite{KC79, EHN96, VV86}.  Note that the regularization error $\alpha^{J+1}\|w\|$ depends on $\delta$ only through $w=G^{-\beta}u$, and $\|G^{-\beta}u\|\ge\|u\|$ for every $\delta$.
\section{Modified iterated Tikhonov--Lavrentiev regularization}\label{sec:mitlar}
\subsection{The algorithm and its structure}
Algorithm~\ref{alg:mitlar} describes the modification of iterated Tikhonov--Lavrentiev regularization studied in this paper.
\begin{algorithm}[Modified iterated Tikhonov--Lavrentiev regularization, Mitlar] \label{alg:mitlar}
Given the filtered data $\bar{u}$, fix the number of updates $J \geq 0$ and the regularization parameter $\alpha > 0$.  Solve for $u_0$ satisfying
\begin{equation}\label{eqn:mitlar:1}
  [(1-\alpha)G+\alpha I]u_0=\bar u.
\end{equation}
Then, for $j=1,\dots,J$, solve for $u_j$ satisfying
\begin{equation}\label{eqn:mitlar:2}
  [(1-\alpha)G+\alpha I](u_j-u_{j-1})=\bar u-Gu_{j-1}.
\end{equation}
\end{algorithm}
For $\alpha>0$ and $\sigma(G)\subset[0,1]$, the operator $(1-\alpha)G+\alpha I$ has spectrum in $[\min(\alpha,1),\max(\alpha,1)]$, so the algorithm is well defined.  The case $J=0$ is the \emph{modified Tikhonov--Lavrentiev} method of \cite{Mays11, IMMS12}.
\begin{definition}\label{def:Dalpha}
  For $\alpha > 0$, define the modified Tikhonov--Lavrentiev operator
\begin{equation*}
  D_\alpha  = [(1-\alpha)G+\alpha I]^{-1}.
\end{equation*}
For $J \ge 0$, define the $J$-th modified iterated Tikhonov--Lavrentiev operator $D_{\alpha,J}$ by $D_{\alpha,J}\bar u = u_J$, where $u_J$ is computed by Algorithm~\ref{alg:mitlar}.
\end{definition}
\begin{proposition}[Closed form] \label{prop:Dalphaj}
For every $j\ge0$,
\begin{equation*}
 D_{\alpha,j}
     =D_\alpha \sum_{i=0}^{j}(I-D_\alpha G)^i, \qquad I-D_\alpha G=\alpha D_\alpha(I-G),
\end{equation*}
and
\begin{equation}\label{eq:DJG}
  D_{\alpha,j}G = I-(I-D_\alpha G)^{j+1}.
\end{equation}
\end{proposition}
\begin{proof}
From \eqref{eqn:mitlar:2}, $u_j = u_{j-1} + D_\alpha(\bar{u} - G u_{j-1}) = D_\alpha \bar{u} + (I - D_\alpha G) u_{j-1}$, with $u_0=D_\alpha\bar u$.  Induction gives the first formula.  The identity $I-D_\alpha G=D_\alpha\big((1-\alpha)G+\alpha I-G\big)=\alpha D_\alpha(I-G)$ is immediate.  Finally, all the operators involved commute, so $D_{\alpha,j}G=\sum_{i=0}^j(I-D_\alpha G)^iD_\alpha G$.  This is a telescoping sum equal to $I-(I-D_\alpha G)^{j+1}$.
\end{proof}

For the differential filter, the operators $D_\alpha$ and $D_\alpha G$ have a simple structure.  It is the source of every improvement over generic Lavrentiev regularization.
\begin{proposition}[Structure of Mitlar]\label{prop:structure}
Let $G=G_\delta$ be the differential filter and let $\alpha>0$.  Then
\begin{equation}\label{eq:structure}
  D_\alpha=(I+\delta^2\Lop)(I+\alpha\delta^2\Lop)^{-1},\qquad D_\alpha G=G_{\sqrt\alpha\,\delta},
\end{equation}
and for every $J\ge0$
\begin{equation}\label{eq:vanCittert}
  D_{\alpha,J}G=\sum_{i=0}^{J}\big(I-G_{\sqrt\alpha\,\delta}\big)^{i}G_{\sqrt\alpha\,\delta}=I-\big(I-G_{\sqrt\alpha\,\delta}\big)^{J+1}.
\end{equation}
Thus, for noise-free data $\bar u=Gu$, $J$ Mitlar steps applied to the filter of radius $\delta$ produce exactly the van Cittert approximate deconvolution of order $J$ \cite{BB98, SA99, DE06, LR12} of the narrower filter $G_{\sqrt\alpha\,\delta}$ applied to $u$.  Equivalently,
\begin{equation}\label{eq:mitlar-symbol}
  D_{\alpha,J}Gu=\sum_k\Big[1-\Big(\frac{\alpha\delta^2\mu_k}{1+\alpha\delta^2\mu_k}\Big)^{J+1}\Big]\hat u_k\varphi_k .
\end{equation}
\end{proposition}
\begin{proof}
Since $G=(I+\delta^2\Lop)^{-1}$, we have $(1-\alpha)G+\alpha I=G\big[(1-\alpha)I+\alpha(I+\delta^2\Lop)\big]=G(I+\alpha\delta^2\Lop)$.  Inverting gives the formula for $D_\alpha$, and then $D_\alpha G=(I+\alpha\delta^2\Lop)^{-1}=G_{\sqrt\alpha\,\delta}$.  Inserting this into \eqref{eq:DJG} gives \eqref{eq:vanCittert}.  Formula \eqref{eq:mitlar-symbol} follows from $1-(1+t)^{-1}=t/(1+t)$ with $t=\alpha\delta^2\mu_k$.
\end{proof}

\begin{remark}[Variational formulation]\label{rem:varmitlar}
By \eqref{eq:structure}, each Mitlar step is one Helmholtz solve with radius $\sqrt\alpha\,\delta$, just as each step of iterated Tikhonov--Lavrentiev is one Helmholtz-type solve.  If $\bar u\in X$, Algorithm~\ref{alg:mitlar} is equivalent to finding $u_j\in X$ such that, for all $v \in X$,
\begin{align}\label{alg:varmitlar}
\alpha \delta^2 \var{\nabla u_0,\nabla v} + \var{u_0,v}
  &=  \delta^2\var{\nabla \bar{u},\nabla v} + \var{\bar{u}, v}, \\
\alpha \delta^2 \var{\nabla u_j,\nabla v} + \var{u_j,v}
  &=  \delta^2\var{\nabla \bar{u},\nabla v} + \var{\bar{u}, v} \nonumber\\
  &\qquad + \alpha \delta^2 \var{\nabla u_{j-1},\nabla v}, \qquad 1\le j\le J. \nonumber
\end{align}
For noise-free data $\bar u=Gu$, the first two terms on each right-hand side equal $(u,v)$ by \eqref{eqnVarDiffFilter}.  For general $\bar u\in L^2$, the right-hand sides are read as the duality pairing $\langle (I+\delta^2\Lop)\bar u,v\rangle$ between $X'$ and $X$.
\end{remark}

We next record operator bounds that hold for every self-adjoint $G$ with spectrum in $[0,1]$.  They therefore apply both to the differential filter and to its discretizations.
\begin{lemma} \label{lem:DGfuncs}
Let $0<\alpha\le1$, and define on $[0,1]$
\begin{equation*}
 f(x)=\frac{1}{(1-\alpha)x+\alpha},\qquad g(x)=x f(x),\qquad 1-g(x)=\alpha(1-x)f(x).
\end{equation*}
Then $f$ is nonincreasing from $f(0)=\frac1\alpha$ to $f(1)=1$; $g$ is increasing from $g(0)=0$ to $g(1)=1$; and $0\le 1-g\le 1$ on $[0,1]$, with $1-g(x)<1$ for $x>0$.
\end{lemma}
\begin{proof}
For $x\in[0,1]$, the number $(1-\alpha)x+\alpha$ is a convex combination of $x$ and $1$, so it lies in $[\alpha,1]$.  It is nondecreasing in $x$, which gives the claims for $f$.  Next, $g'(x)=\alpha\,((1-\alpha)x+\alpha)^{-2}>0$, so $g$ is strictly increasing from $g(0)=0$ to $g(1)=1$.  The claims for $1-g$ follow.
\end{proof}
\begin{lemma}\label{lem:bounds}
Let $G$ be self-adjoint with $\sigma(G)\subset[0,1]$ and let $0<\alpha\le1$.  Then
\begin{equation} \label{eqnDalphaBound}
  \| D_\alpha \| \leq \frac{1}{\alpha}, \quad \| D_\alpha G \| \leq 1, \quad  \| I - D_\alpha G \| \leq 1,\quad \|D_{\alpha,J}\|\le\frac{J+1}{\alpha}.
\end{equation}
If $0\in\sigma(G)$, which holds for every compact $G$ on an infinite-dimensional space and in particular for the differential filter, then $\|D_{\alpha,J}\|=(J+1)/\alpha$.  So the noise-amplification bound cannot be improved.
\end{lemma}
\begin{proof}
By the spectral theorem, $\|\phi(G)\|=\sup_{\sigma(G)}|\phi|$ for $\phi\in\{f,g,1-g\}$, and Lemma~\ref{lem:DGfuncs} gives the first three bounds.  By Proposition~\ref{prop:Dalphaj}, $D_{\alpha,J}=q_J(G)$ with $q_J(x)=f(x)\sum_{i=0}^J(1-g(x))^i$.  Each summand is a product of the nonnegative, nonincreasing functions $f$ and $(1-g)^i$.  Hence $q_J$ is nonincreasing on $[0,1]$ and $\sup_{[0,1]}q_J=q_J(0)=(J+1)/\alpha$, with equality in the supremum over $\sigma(G)$ whenever $0\in\sigma(G)$.
\end{proof}

\subsection{Error analysis: noise-free data}
\begin{theorem}[Noise-free error]\label{thm:noisefree}
Let $G=G_\delta$, $\alpha>0$, $J\ge0$, $u\in L^2(\Omega)^d$, and $e_J:=u-D_{\alpha,J}Gu$.
\begin{enumerate}[(a)]
  \item \emph{Representation.}  $e_J=(I-G_{\sqrt\alpha\,\delta})^{J+1}u$.  If $u\in\mathcal D(\Lop^{J+1})$, then
  \begin{equation} \label{eqn:continuouserrorequation}
  e_J = (-\alpha \delta^2)^{J+1} (D_\alpha G)^{J+1}\Delta^{J+1}u .
  \end{equation}
  \item \emph{Bounds on a Hilbert scale.}  If $u\in\mathcal D(\Lop^s)$ for some $0\le s\le J+1$, then, with $0^0:=1$,
  \begin{equation}\label{eq:scalebound}
    \|e_J\|\le\Big(\frac{s}{J+1}\Big)^{s}(\alpha\delta^2)^{s}\,\|\Lop^{s}u\|.
  \end{equation}
  In particular, for $u\in\mathcal D(\Lop^{J+1})$,
  \begin{equation} \label{eqn:continuouserrorbound}
  \| e_J \| \leq (\alpha \delta^2)^{J+1} \| \Delta^{J+1} u\|,
  \end{equation}
  and for every $u\in X$, $\|e_J\|\le\big(\alpha\delta^2/(2(J+1))\big)^{1/2}\|\nabla u\|$.
  \item \emph{Sharpness and saturation.}  For every $u\in L^2(\Omega)^d$, the quotient $(\alpha\delta^2)^{-(J+1)}\|e_J\|$ increases monotonically to $\|\Lop^{J+1}u\|\in[0,\infty]$ as $\alpha\delta^2\downarrow0$.  Consequently \eqref{eqn:continuouserrorbound} is asymptotically exact.  Moreover, $\|e_J\|=O((\alpha\delta^2)^{J+1})$ holds if and only if $u\in\mathcal D(\Lop^{J+1})$, and $\|e_J\|=o((\alpha\delta^2)^{J+1})$ holds only for $u=0$.
  \item \emph{Convergence.}  For every $u\in L^2(\Omega)^d$, $\|e_J\|\to0$ as $\alpha\delta^2\to0$ with $J$ fixed, and as $J\to\infty$ with $\alpha,\delta$ fixed.
\end{enumerate}
\end{theorem}
\begin{proof}
(a) The first formula is \eqref{eq:vanCittert}.  On $\mathcal D(\Lop)$ we have $I-G_{\sqrt\alpha\delta}=\alpha\delta^2\Lop G_{\sqrt\alpha\delta}=-\alpha\delta^2 (D_\alpha G)\Delta$, and $G_{\sqrt\alpha\delta}$ commutes with $\Lop$ on $\mathcal D(\Lop^{J+1})$.  This gives \eqref{eqn:continuouserrorequation}.

(b) Write $t_k=\alpha\delta^2\mu_k$.  By \eqref{eq:mitlar-symbol}, $\|e_J\|^2=\sum_k m(t_k)^2|\hat u_k|^2$ with $m(t)=(t/(1+t))^{J+1}$.  We claim that
\begin{equation}\label{eq:scalar}
  m(t)\le\Big(\frac{s}{J+1}\Big)^{s}t^{s}\qquad (t\ge0,\ 0\le s\le J+1).
\end{equation}
Put $r=t/(1+t)\in[0,1)$, so that $t=r/(1-r)$ and $m(t)t^{-s}=r^{a}(1-r)^{s}$ with $a=J+1-s\ge0$.  The maximum of $r^a(1-r)^s$ over $[0,1]$ is $a^as^s/(a+s)^{a+s}\le\big(s/(a+s)\big)^s=(s/(J+1))^s$.  Therefore $\|e_J\|^2\le (s/(J+1))^{2s}(\alpha\delta^2)^{2s}\sum_k\mu_k^{2s}|\hat u_k|^2$, which is \eqref{eq:scalebound}.  The case $s=J+1$ is \eqref{eqn:continuouserrorbound}, and the case $s=\frac12$, together with $\|\Lop^{1/2}u\|=\|\nabla u\|$, gives the last bound.

(c) Let $\varepsilon=\alpha\delta^2$.  Then $\varepsilon^{-(J+1)}m(\varepsilon\mu_k)=(\mu_k/(1+\varepsilon\mu_k))^{J+1}$ increases to $\mu_k^{J+1}$ as $\varepsilon\downarrow0$.  The monotone convergence theorem gives $\varepsilon^{-2(J+1)}\|e_J\|^2\uparrow\sum_k\mu_k^{2(J+1)}|\hat u_k|^2=\|\Lop^{J+1}u\|^2$.  The limit is finite if and only if $u\in\mathcal D(\Lop^{J+1})$, and it is zero if and only if $u=0$.

(d) The multiplier $m(t_k)$ is bounded by $1$.  It tends to $0$ for each $k$ as $\alpha\delta^2\to0$, and also as $J\to\infty$ because $t_k/(1+t_k)<1$.  Dominated convergence gives the claim.
\end{proof}

\begin{remark}
Part (c) shows that the order $(\alpha\delta^2)^{J+1}$ in \eqref{eqn:continuouserrorbound} cannot be improved, and that each update step gains exactly one factor $\alpha\delta^2$.  By contrast, each step of iterated Tikhonov--Lavrentiev regularization gains at most a factor $\alpha$; see Theorem~\ref{thm:ITL} and the next result.  Part (b) also covers rough data: the bound decays like $(J+1)^{-s}$ when $s<J+1$, so increasing $J$ helps even if $u$ is not smooth.  The compatibility conditions hidden in $\mathcal D(\Lop^{J+1})$ are genuine restrictions, and for $s\le\frac12$ they are void.
\end{remark}

The next theorem compares Mitlar with iterated Tikhonov--Lavrentiev regularization (Definition~\ref{def:ITL}) for an arbitrary filter.
\begin{theorem}[Mitlar is never worse than iterated Lavrentiev]\label{thm:domination}
Let $G$ be self-adjoint with $\sigma(G)\subset[0,1]$, let $\alpha>0$ and $J\ge0$, and let $u_J^{\rm M}$ and $u_J^{\rm ITL}$ be the $J$-th Mitlar and iterated Tikhonov--Lavrentiev iterates computed from the noise-free data $\bar u=Gu$ with the same $\alpha$.  Then for every $u$
\begin{equation}\label{eq:domination}
  \|u-u_J^{\rm M}\|\;\le\;\|u-u_J^{\rm ITL}\|,\qquad \|u-u_J^{\rm ITL}\|\;\ge\;\Big(\frac{\alpha}{1+\alpha}\Big)^{J+1}\|u\|.
\end{equation}
\end{theorem}
\begin{proof}
For iterated Tikhonov--Lavrentiev, $e_j=e_{j-1}-(G+\alpha I)^{-1}Ge_{j-1}=\alpha(G+\alpha I)^{-1}e_{j-1}$ and $e_0=\alpha(G+\alpha I)^{-1}u$.  So $u-u_J^{\rm ITL}=\psi(G)^{J+1}u$ with $\psi(x)=\alpha/(x+\alpha)$.  By \eqref{eq:DJG}, $u-u_J^{\rm M}=\chi(G)^{J+1}u$ with $\chi(x)=1-g(x)=\alpha(1-x)/((1-\alpha)x+\alpha)$.  For $x\in[0,1]$ both multipliers are nonnegative.  Moreover, $\chi(x)\le\psi(x)$ is equivalent to $(1-x)(x+\alpha)\le(1-\alpha)x+\alpha$, that is, to $-x^2\le0$.  Hence $0\le\chi^{J+1}\le\psi^{J+1}$ on $\sigma(G)$.  The spectral theorem gives $\|\chi(G)^{J+1}u\|^2=\int\chi^{2(J+1)}\,d(E_xu,u)\le\int\psi^{2(J+1)}\,d(E_xu,u)$, where $E_x$ is the spectral family of $G$.  Finally, $\psi(x)\ge\alpha/(1+\alpha)$ on $[0,1]$.
\end{proof}
\begin{remark}
For the differential filter, the lower bound in \eqref{eq:domination} does not depend on $\delta$.  Iterated Tikhonov--Lavrentiev does not become exact as the filter becomes weak, whereas Mitlar does: $\|u-u^{\rm M}_J\|\le(\alpha\delta^2)^{J+1}\|\Delta^{J+1}u\|$.  Theorem~\ref{thm:domination} applies verbatim to the discrete filters of Section~\ref{sec:numerics}, which are symmetric positive definite matrices with spectrum in $(0,1]$.  Since the multipliers $\chi^{J+1}$ and $\psi^{J+1}$ are nonincreasing in $J$, the noise-free Mitlar error with $J$ updates is also bounded by the error of plain and modified Tikhonov--Lavrentiev regularization.
\end{remark}

\subsection{Error analysis: noisy data}
Noise amplification is one of the fundamental difficulties of ill-posed problems \cite{EHN96}.  We now consider Problem~\ref{prob:noisyDiff}, in which the data $\bar u=Gu-\epsilon$ contain additive noise.
\begin{theorem}\label{thmContErrorNoise}
Let $G=G_\delta$, $0<\alpha\le1$, $J\ge0$, $\bar u=Gu-\epsilon$ with $\|\epsilon\|\le\epsilon_0$, and let $u_J=D_{\alpha,J}\bar u$.  Then the error $e_J=u-u_J$ satisfies
\begin{equation} \label{eqn:errorequation}
  e_J = (I-G_{\sqrt\alpha\,\delta})^{J+1}u + D_\alpha \sum_{j=0}^{J}(I - D_\alpha G)^j \epsilon .
\end{equation}
If $u\in\mathcal D(\Lop^s)$ with $0\le s\le J+1$, then
\begin{equation} \label{eqn:errorbound}
  \| e_J \| \leq \Big(\frac{s}{J+1}\Big)^{s}(\alpha \delta^2)^{s} \| \Lop^{s} u\| + \frac{(J+1)\epsilon_0}{\alpha},
\end{equation}
and in particular $\| e_J \| \leq (\alpha \delta^2)^{J+1} \| \Delta^{J+1} u\| + (J+1)\epsilon_0/\alpha$ for $u\in\mathcal D(\Lop^{J+1})$.  The constant $(J+1)/\alpha$ in front of $\epsilon_0$ cannot be improved.
\end{theorem}
\begin{proof}
By linearity, $u_J=D_{\alpha,J}Gu-D_{\alpha,J}\epsilon$, so $e_J=(u-D_{\alpha,J}Gu)+D_{\alpha,J}\epsilon$.  Proposition~\ref{prop:structure} identifies the first term and Proposition~\ref{prop:Dalphaj} the second, which proves \eqref{eqn:errorequation}.  The bound follows from Theorem~\ref{thm:noisefree}(b) and $\|D_{\alpha,J}\|\le(J+1)/\alpha$ (Lemma~\ref{lem:bounds}), which is sharp by the same lemma.
\end{proof}
Each update step improves the regularization error by a factor $\alpha\delta^2$, while the bound on the propagated noise grows by $\epsilon_0/\alpha$.  The two effects can be balanced in two ways: through $\alpha$ for fixed $J$, or through $J$ for fixed $\alpha$.
\begin{corollary}[A priori choice of $\alpha$]\label{cor:alpha}
In the setting of Theorem~\ref{thmContErrorNoise}, let $0<s\le J+1$ and $N_s:=\|\Lop^s u\|>0$, and assume $(J+1)\epsilon_0\le\delta^{2s}N_s$.  Then the choice
\begin{equation*}
  \alpha=\Big(\frac{(J+1)\epsilon_0}{\delta^{2s}N_s}\Big)^{1/(s+1)}\in(0,1]
\end{equation*}
gives
\begin{equation*}
  \|e_J\|\le 2\,\delta^{2s/(s+1)}\,N_s^{1/(s+1)}\,\big((J+1)\epsilon_0\big)^{s/(s+1)} .
\end{equation*}
\end{corollary}
\begin{proof}
Insert $\alpha$ into \eqref{eqn:errorbound} and use $(s/(J+1))^s\le1$.  Both terms are then equal to $\delta^{2s/(s+1)}N_s^{1/(s+1)}((J+1)\epsilon_0)^{s/(s+1)}$.
\end{proof}
For $s=J+1$, Corollary~\ref{cor:alpha} gives the same order $\epsilon_0^{(J+1)/(J+2)}$ in the noise level as Theorem~\ref{thm:ITL}.  It also gives an explicit factor $\delta^{2(J+1)/(J+2)}$, which tends to zero with the filter radius.
\begin{corollary}[The number of updates as regularization parameter]\label{cor:J}
In the setting of Theorem~\ref{thmContErrorNoise}, fix $0<\alpha\le1$, let $s>0$ and $N_s=\|\Lop^s u\|>0$, and set $A=(s\alpha\delta^2)^sN_s$ and $B=\epsilon_0/\alpha$.  If $x:=(A/B)^{1/(s+1)}\ge\max(s,1)$, then the choice $J+1=\lceil x\rceil$ gives
\begin{equation*}
  \|e_J\|\le 3\,(s\,\delta^2\epsilon_0)^{s/(s+1)}N_s^{1/(s+1)} .
\end{equation*}
\end{corollary}
\begin{proof}
Let $n=J+1=\lceil x\rceil$.  Then $x\le n\le 2x$ because $x\ge1$, and $n\ge s$, so \eqref{eqn:errorbound} applies and gives $\|e_J\|\le An^{-s}+Bn$.  We have $An^{-s}\le Ax^{-s}=A^{1/(s+1)}B^{s/(s+1)}$ and $Bn\le2Bx=2A^{1/(s+1)}B^{s/(s+1)}$.  Finally, $A^{1/(s+1)}B^{s/(s+1)}=(s\delta^2\epsilon_0)^{s/(s+1)}N_s^{1/(s+1)}$.
\end{proof}
When the number of updates is the regularization parameter, the order $\epsilon_0^{s/(s+1)}$ is achieved for every smoothness index $s>0$; the saturation at $s=J+1$ in Corollary~\ref{cor:alpha} disappears.  In practice $\epsilon_0$ is known but $N_s$ is not.  Section~\ref{sec:descent} gives a computable a posteriori rule for choosing $J$.
\subsection{Discrete Mitlar applied to the differential filter}
We now extend the analysis to the finite element discretization of Mitlar.  Let $\Lop_h:X^h\to X^h$ be the discrete Dirichlet Laplacian, $(\Lop_hw^h,v^h)=(\nabla w^h,\nabla v^h)$, and let $P_h$ be the $L^2$ projection onto $X^h$.  Then $G^h=(I+\delta^2\Lop_h)^{-1}P_h$.  On $X^h$, the restriction of $G^h$ is self-adjoint with spectrum in $(0,1)$ (Lemma~\ref{lemDiscFilterSelfAdjoint}).  Define $D^h_\alpha=[(1-\alpha)G^h+\alpha I]^{-1}$ and $D^h_{\alpha,J}$ on $X^h$ as in Definition~\ref{def:Dalpha}.  By Lemma~\ref{lem:bounds} applied to $G^h|_{X^h}$, for $0<\alpha\le1$,
\begin{equation}\label{eq:discbounds}
  \|G^h\|\le1,\ \ \|D^h_\alpha\|\le\frac1\alpha,\ \ \|D^h_\alpha G^h\|\le1,\ \ \|I-D^h_\alpha G^h\|\le1,\ \ \|D^h_{\alpha,J}\|\le\frac{J+1}{\alpha}.
\end{equation}
The computation in Proposition~\ref{prop:structure} carries over verbatim: on $X^h$, $D^h_\alpha G^h=(I+\alpha\delta^2\Lop_h)^{-1}$.
\begin{algorithm} [Discrete modified iterated Tikhonov--Lavrentiev regularization]
\label{algo:varmitlardisc}
Let $\bar u^h\in X^h$ be given data; for noise-free data, $\bar u^h=G^hu$.  Choose $\alpha>0$ and $J\ge0$, and define $u_j^h = D_{\alpha,j}^h \bar{u}^h\in X^h$ recursively as the unique solutions of
\begin{align} \label{alg:varmitlardisc}
\alpha \delta^2 \var{\nabla u_0^h,\nabla v^h} + \var{u_0^h,v^h}
  &=  \delta^2\var{\nabla\bar u^h,\nabla v^h}+\var{\bar u^h,v^h},\\
\alpha \delta^2 \var{\nabla u_j^h,\nabla v^h} + \var{u_j^h,v^h}
  &=  \delta^2\var{\nabla\bar u^h,\nabla v^h}+\var{\bar u^h,v^h} \nonumber\\
  &\qquad + \alpha \delta^2 \var{\nabla u_{j-1}^h,\nabla v^h}, \nonumber
\end{align}
for all $v^h \in X^h$ and $1\le j\le J$.  For $\bar u^h=G^hu$, the data terms equal $(u,v^h)$ by Definition~\ref{defn:discdiffFilterVar}.
\end{algorithm}
We measure discrete errors in the energy norm $|||w|||_\alpha^2:=\alpha\delta^2\|\nabla w\|^2+\|w\|^2$ associated with the bilinear form $a(w,v)=\alpha\delta^2(\nabla w,\nabla v)+(w,v)$.
\begin{theorem}\label{thm:discerror}
Let $\alpha>0$, $J\geq 0$, $u\in L^2(\Omega)^d$, and $u_j:=D_{\alpha,j}Gu$ for $0\le j\le J$.  Then $u_j\in X$ and
\begin{equation}\label{eq:discsum}
\begin{aligned}
  \| D_{\alpha,J}Gu-D_{\alpha,J}^hG^hu\|&\le|||D_{\alpha,J}Gu-D_{\alpha,J}^hG^hu|||_\alpha\\
  &\le\sum_{j=0}^{J}\inf_{w^h\in X^h}|||u_j-w^h|||_\alpha .
\end{aligned}
\end{equation}
Consequently, if $u_j\in H^{k+1}(\Omega)^d$ for $0\le j\le J$ and $u\in\mathcal D(\Lop^s)$ for some $0\le s\le J+1$, then
\begin{align}\label{eq:discfinal}
  \| u - D_{\alpha, J}^h G^h u \| &\leq \Big(\frac{s}{J+1}\Big)^s(\alpha \delta^2)^{s} \| \Lop^{s} u\| \\
      & \quad + C\,(J+1)\,(\sqrt{\alpha}\, \delta h^k + h^{k+1}) \max_{0\leq j \leq J} | D_{\alpha,j} G u |_{H^{k+1}(\Omega)},\nonumber
\end{align}
where $C$ is the constant in \eqref{eqn:approxInequality} and does not depend on $h$, $\alpha$, $\delta$, $J$ or $u$.
\end{theorem}
\begin{proof}
By \eqref{eq:vanCittert}, $u_j=\big(I-(I-G_{\sqrt\alpha\delta})^{j+1}\big)u$ is a polynomial in $G_{\sqrt\alpha\delta}$ without constant term, so $u_j\in X$.  By Remark~\ref{rem:varmitlar} with $\bar u=Gu$, and by Algorithm~\ref{algo:varmitlardisc} with $\bar u^h=G^hu$, we have for all $v\in X$ and $v^h\in X^h$
\begin{align*}
 a(u_0,v)=(u,v),&\qquad a(u_j,v)=(u,v)+\alpha\delta^2(\nabla u_{j-1},\nabla v),\\
 a(u^h_0,v^h)=(u,v^h),&\qquad a(u^h_j,v^h)=(u,v^h)+\alpha\delta^2(\nabla u^h_{j-1},\nabla v^h).
\end{align*}
Let $E_j=u_j-u_j^h$.  Subtracting gives the error equations
\begin{equation}\label{eqn:varerror}
  a(E_0,v^h)=0,\qquad a(E_j,v^h)=\alpha\delta^2(\nabla E_{j-1},\nabla v^h)\qquad\forall v^h\in X^h .
\end{equation}
Let $R_h:X\to X^h$ be the $a$-orthogonal projection, and split $E_j=\rho_j+\theta_j$ with $\rho_j=u_j-R_hu_j$ and $\theta_j=R_hu_j-u_j^h\in X^h$.  Then $|||\rho_j|||_\alpha=\inf_{w^h\in X^h}|||u_j-w^h|||_\alpha$.  Since $a(\rho_j,v^h)=0$, \eqref{eqn:varerror} gives $\theta_0=0$ and, for $j\ge1$,
\begin{align*}
  |||\theta_j|||_\alpha^2=a(E_j,\theta_j)&=\alpha\delta^2(\nabla E_{j-1},\nabla\theta_j)\\
  &\le\big(\sqrt\alpha\,\delta\|\nabla E_{j-1}\|\big)\big(\sqrt\alpha\,\delta\|\nabla\theta_j\|\big)\le|||E_{j-1}|||_\alpha\,|||\theta_j|||_\alpha .
\end{align*}
Hence $|||E_j|||_\alpha\le|||\rho_j|||_\alpha+|||E_{j-1}|||_\alpha$ and $|||E_0|||_\alpha=|||\rho_0|||_\alpha$.  Summing proves \eqref{eq:discsum}, since $\|E_J\|\le|||E_J|||_\alpha$.  For \eqref{eq:discfinal}, write $u-D^h_{\alpha,J}G^hu=(u-D_{\alpha,J}Gu)+E_J$.  Bound the first term by Theorem~\ref{thm:noisefree}(b).  For the second, choose $w^h$ in \eqref{eq:discsum} as in \eqref{eqn:approxInequality}, which gives $|||u_j-w^h|||_\alpha\le C(\sqrt\alpha\,\delta h^k+h^{k+1})|u_j|_{H^{k+1}}$.
\end{proof}
\begin{remark}
A direct energy argument based on Young's inequality yields a recursion of the form $|||E_j|||^2_\alpha\le c_1|||\rho_j|||^2_\alpha+c_2|||E_{j-1}|||^2_\alpha$ with $c_2>1$.  That produces constants growing like $c_2^{J}$.  The Ritz-projection argument above shows that the discretization errors of the individual steps simply add up.
\end{remark}
\begin{corollary}[Piecewise linear elements on convex domains]\label{cor:P1}
Let $\Omega$ be convex and $k=1$, and let $u\in X$ with $\Delta u\in L^2(\Omega)^d$.  Then for every $0\le s\le J+1$ with $u\in\mathcal D(\Lop^s)$,
\begin{equation*}
  \| u - D_{\alpha, J}^h G^h u \| \le \Big(\frac{s}{J+1}\Big)^s(\alpha \delta^2)^{s} \| \Lop^{s} u\| + C_\Omega\,(J+1)\,(\sqrt{\alpha}\, \delta h + h^{2})\,\|\Delta u\|,
\end{equation*}
with $C_\Omega$ independent of $h,\alpha,\delta,J$ and $u$.
\end{corollary}
\begin{proof}
On a convex domain, $\mathcal D(\Lop)=H^2(\Omega)^d\cap X$ and $|w|_{H^2}\le C\|\Delta w\|$ for $w\in\mathcal D(\Lop)$ \cite{Gri85}.  Since $u_j=\big(I-(I-G_{\sqrt\alpha\delta})^{j+1}\big)u$ and the multiplier lies in $[0,1]$, we get $u_j\in\mathcal D(\Lop)$ and $\|\Delta u_j\|=\|\Lop u_j\|\le\|\Lop u\|=\|\Delta u\|$.  Apply Theorem~\ref{thm:discerror}.
\end{proof}

Finally we turn to Problem~\ref{prob:noisyDiff} in the discrete setting.  Suppose the data consist of discrete measurements containing noise, $\bar u^h=G^hu-\varepsilon$ with $\varepsilon\in X^h$.  We then apply Algorithm~\ref{algo:varmitlardisc} to these noisy data.
\begin{theorem}\label{thm:Noisediscerror}
  Let $0<\alpha\le1$, and let the noise $\varepsilon \in X^h$ satisfy $\| \varepsilon \| \leq \epsilon_0$.  Let $\bar u^h=G^hu-\varepsilon$.  Under the assumptions of Theorem~\ref{thm:discerror}, the error $e_J^h:=u - D_{\alpha,J}^h \bar{u}^h$ satisfies
\begin{align*}
\|e_J^h\| &\leq \frac{J+1}{\alpha} \epsilon_0+ \Big(\frac{s}{J+1}\Big)^s(\alpha \delta^2)^{s} \| \Lop^{s} u\| \\
&\qquad + C\,(J+1)\, (\sqrt{\alpha}\, \delta h^k + h^{k+1}) \max_{0\leq j \leq J} | D_{\alpha,j} G u |_{H^{k+1}(\Omega)}.
\end{align*}
\end{theorem}
\begin{proof}
By linearity, $u - D_{\alpha,J}^h \bar{u}^h=(u - D_{\alpha,J}^hG^hu) + D_{\alpha,J}^h \varepsilon$.  Bound the first term by Theorem~\ref{thm:discerror} and the second by \eqref{eq:discbounds}.
\end{proof}
\section{Descent properties of Mitlar and a stopping rule}\label{sec:descent}
In this section $G$ is any self-adjoint, injective operator with $\sigma(G)\subset[0,1]$; in particular $(Gv,v)>0$ for $v\ne0$.  Solving Problem~\ref{prob:noisefreeDiff} is equivalent to minimizing
\begin{equation*}
 E_{0}(v)\coloneqq \frac{1}{2}(G v,v) - (\bar u,v),\qquad v\in L^2 .
\end{equation*}
Indeed, if $\bar u=Gu$, then
\begin{equation}\label{eq:energyerror}
  E_0(v)-E_0(u)=\tfrac12\|G^{1/2}(v-u)\|^2 .
\end{equation}
So the energy measures the error in the weighted norm $\|G^{1/2}\cdot\|$.  For the differential filter, this norm is $((I-\delta^2\Delta)^{-1}w,w)^{1/2}$, which is weaker than the $L^2$ norm on scales below $\delta$.  We write $d_j:=u_{j+1}-u_j$ for the Mitlar updates.
\begin{proposition}\label{5.1:prop1}
Let $\alpha>0$ and $\bar u=Gu$.  The Mitlar iterates satisfy
\begin{equation}\label{eqn:EnergyNorm}
E_{0}(u_j) - E_{0}(u_{j+1})=\big([(\tfrac{1}{2}-\alpha)G+\alpha I]d_j,d_j\big)\geq \min(\alpha,\tfrac12)\,\|d_j\|^2 .
\end{equation}
Hence $E_0(u_{j+1})<E_0(u_j)$ unless $u_{j+1}=u_j$.  Moreover $E_0(u_j)\downarrow E_0(u)=\min E_0$ and $\|u-u_j\|\to0$ as $j\to\infty$, so $\{u_j\}$ is a minimizing sequence.
\end{proposition}
\begin{proof}
Using the self-adjointness of $G$ and \eqref{eqn:mitlar:2} in the form $\bar u-Gu_j=[(1-\alpha)G+\alpha I]d_j$, we get
\begin{equation*}
E_0(u_j)-E_0(u_j+d_j)=(\bar u-Gu_j,d_j)-\tfrac12(Gd_j,d_j)=\big([(\tfrac12-\alpha)G+\alpha I]d_j,d_j\big).
\end{equation*}
The function $x\mapsto(\frac12-\alpha)x+\alpha$ is affine, equals $\alpha$ at $x=0$ and $\frac12$ at $x=1$, and is therefore at least $\min(\alpha,\frac12)$ on $\sigma(G)\subset[0,1]$.  This proves \eqref{eqn:EnergyNorm}.  By \eqref{eq:DJG}, $u-u_j=\chi(G)^{j+1}u$ with $\chi(x)=1-g(x)=\alpha(1-x)/((1-\alpha)x+\alpha)$.  For every $\alpha>0$ and $x\in[0,1]$ the denominator is positive, so $\chi(x)\ge0$.  Moreover, $\chi(x)<1$ if and only if $\alpha(1-x)<(1-\alpha)x+\alpha$, that is, if and only if $x>0$.  Since $G$ is injective, the spectral measure of $u$ does not charge $\{0\}$, and dominated convergence gives $\|u-u_j\|\to0$.  Then $E_0(u_j)\to E_0(u)$ by \eqref{eq:energyerror}.
\end{proof}
Proposition~\ref{5.1:prop1} extends the descent property from $0<\alpha\le\frac12$ to all $\alpha>0$.

By \eqref{eqn:mitlar:2}, $u_{j+1}=u_j$ implies $Gu_j=\bar u$.  In the noisy case such convergence is not desired, because $\bar u=Gu-\epsilon$.  If the iterates converge, their limit is $G^{-1}\bar u=u-G^{-1}\epsilon$, the unstable naive inverse.  If $\bar u\notin\operatorname{Range}(G)$, then $\|u_j\|\to\infty$.  It is therefore essential to stop after finitely many steps.  Since $\bar u+\epsilon=Gu$, the relevant functional is
\begin{equation*}
E_{\epsilon}(v)\coloneqq \frac{1}{2}(G v,v) - (\bar u + \epsilon,v),\qquad E_\epsilon(v)-E_\epsilon(u)=\tfrac12\|G^{1/2}(v-u)\|^2 .
\end{equation*}
Expanding as in the proof of Proposition~\ref{5.1:prop1} gives
\begin{equation}\label{eq:noisyenergy}
E_{\epsilon}(u_j)-E_{\epsilon}(u_{j+1}) = (\epsilon,d_j) + \big([(\tfrac{1}{2}-\alpha)G+\alpha I]d_j,d_j\big).
\end{equation}
\begin{theorem}\label{5.3:them1}
Let $\alpha>0$, $\bar u=Gu-\epsilon$ with $\|\epsilon\|\le\epsilon_0$, and let $u_j$ be the Mitlar iterates computed from $\bar u$.  Set $\tau_\alpha:=\epsilon_0/\min(\alpha,\frac12)$.
\begin{enumerate}[(a)]
\item If
\begin{equation} \label{alphabound}
\|u_{j+1}-u_j\| \ge \tau_\alpha,\quad\text{i.e., for } 0<\alpha\le\tfrac12,\quad \frac{\epsilon_0}{\norm{u_{j+1}-u_j}} \leq \alpha,
\end{equation}
then $E_\epsilon(u_{j+1})\le E_\epsilon(u_j)$; equivalently, $\|G^{1/2}(u-u_{j+1})\|\le\|G^{1/2}(u-u_j)\|$.  The inequalities are strict if the inequality in \eqref{alphabound} is strict.
\item The update norms $\|u_{j+1}-u_j\|$ are nonincreasing in $j$ and tend to $0$ as $j\to\infty$.
\item Consequently, if $\epsilon_0>0$, the stopping index $j_*:=\min\{j\ge0:\|u_{j+1}-u_j\|<\tau_\alpha\}$ is finite.  Moreover, the weighted error $\|G^{1/2}(u-u_j)\|$ is nonincreasing for $0\le j\le j_*$.
\end{enumerate}
\end{theorem}
\begin{proof}
(a) By the Cauchy--Schwarz inequality, $(\epsilon,d_j)\ge-\epsilon_0\|d_j\|$.  As in the proof of Proposition~\ref{5.1:prop1}, the quadratic term in \eqref{eq:noisyenergy} is at least $\min(\alpha,\frac12)\|d_j\|^2$.  Hence $E_\epsilon(u_j)-E_\epsilon(u_{j+1})\ge\|d_j\|\big(\min(\alpha,\frac12)\|d_j\|-\epsilon_0\big)\ge0$ under \eqref{alphabound}.

(b) The residual $r_j=\bar u-Gu_j$ satisfies $r_{j+1}=r_j-GD_\alpha r_j=(I-GD_\alpha)r_j$, and $d_j=D_\alpha r_j$.  Hence $d_j=(I-D_\alpha G)^{j+1}D_\alpha\bar u=\phi_j(G)\bar u$ with $\phi_j=f\,\chi^{j+1}$, where $f(x)=((1-\alpha)x+\alpha)^{-1}\le 1/\min(\alpha,1)$ and $\chi=1-xf$ is as in the proof of Proposition~\ref{5.1:prop1}.  There we showed that $0\le\chi\le1$ on $[0,1]$ and $\chi<1$ on $(0,1]$.  So $\phi_j$ decreases pointwise in $j$ and tends to $0$ on $(0,1]$.  Since $G$ is injective, the spectral measure of $\bar u$ does not charge $\{0\}$.  Monotone and dominated convergence give both claims.

(c) Finiteness of $j_*$ follows from (b).  For $j<j_*$ we have $\|d_j\|\ge\tau_\alpha$, and (a) applies.
\end{proof}
Theorem~\ref{5.3:them1} gives a computable stopping rule: iterate while the update is at least $\tau_\alpha$, which equals $\epsilon_0/\alpha$ when $\alpha\le\frac12$.  Up to the stopping index, every update moves the iterate closer to the true solution $u$ in the weighted norm $\|G^{1/2}\cdot\|$.  The rule is conservative: $j_*$ never exceeds the first index at which $j\mapsto E_\epsilon(u_j)$ increases.  Since the condition \eqref{alphabound} is only sufficient, the rule may stop before the minimum of the energy.

The rule can be refined if more is known about the noise.  Suppose $P$ is an orthogonal projection with $P\epsilon=0$; for example, $P$ may project onto a band of scales that is known to be noise free.  Then $(\epsilon,Pd_j)=0$, so the projected update $Pd_j$ is uncorrelated with the noise.  This suggests the following heuristic modification, which we do not analyze further.
\begin{algorithm}\label{alg:projected}
Let $\bar u = Gu - \epsilon$ with $\| \epsilon \| \leq \epsilon_0$, and let $P$ be an orthogonal projection with $P\epsilon=0$.  Fix $\alpha>0$ and $J \geq 0$.  Compute $u_0$ from \eqref{eqn:mitlar:1}.  For $j=1, \dots, J$, compute the Mitlar update $d_{j-1}$ from \eqref{eqn:mitlar:2}.  If $\|d_{j-1}\|\ge\tau_\alpha$, set $u_j=u_{j-1}+d_{j-1}$.  Otherwise, either increase $\alpha$ so that \eqref{alphabound} holds and recompute, or set $u_j=u_{j-1}+Pd_{j-1}$.
\end{algorithm}
\section{Numerical illustrations}\label{sec:numerics}
We present three sets of experiments.  In Section~\ref{sec:numericalstopping} we test the stopping criterion.  In Section~\ref{numericalcomparison} we compare four members of the Tikhonov--Lavrentiev family on the deconvolution of the differential filter: Tikhonov--Lavrentiev, iterated Tikhonov--Lavrentiev, modified Tikhonov--Lavrentiev, and Mitlar.  In Section~\ref{numericalconvergence} we compute and compare the convergence rates of these four methods.
\subsection{An example of the stopping criterion}
\label{sec:numericalstopping}
We implemented the stopping criterion of Theorem~\ref{5.3:them1} in MATLAB (version R2014b).  The true solution is $u = \sin(\pi x) + \sin(200\pi x)$ on the interval $[0, 2]$, plotted in Figure~\ref{fig:stoppingtrue}.
\begin{figure}
  \centering
    \includegraphics[width = 0.75\textwidth]{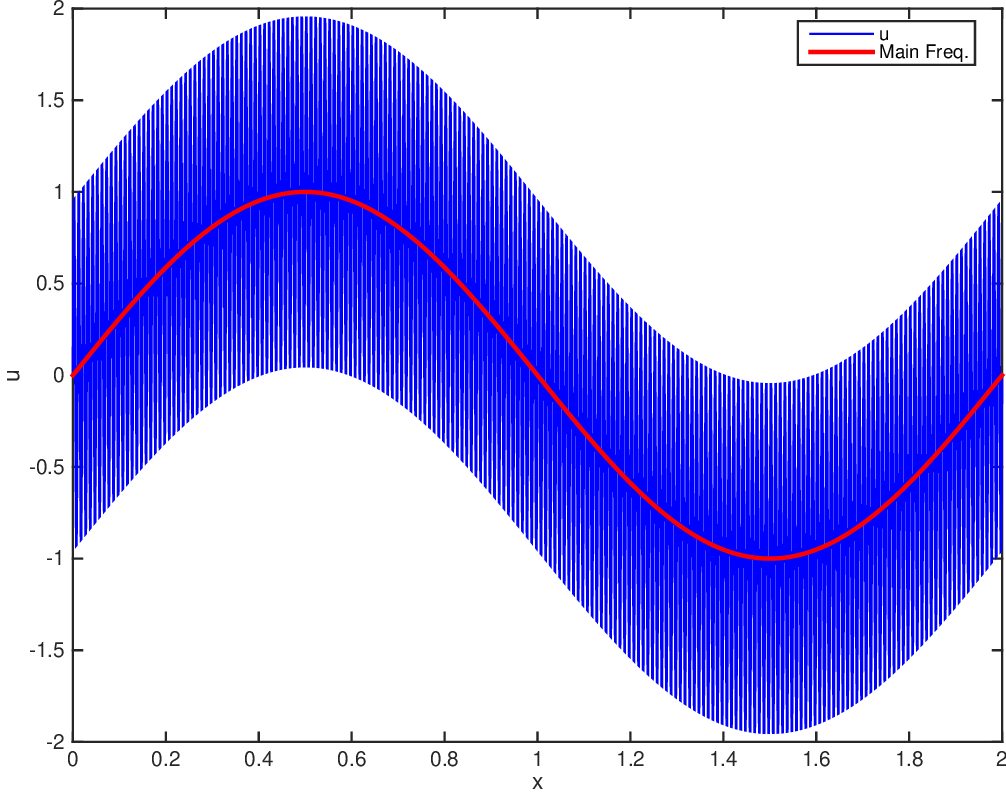}                   
    \caption{$u = \sin(\pi x) + \sin(200\pi x)$, with high and low frequencies, used to test the stopping criterion.}
    \label{fig:stoppingtrue}
\end{figure}
We discretize the interval with step size $h = \frac{2}{1000}$, so there are $1001$ sample points, and set the filter radius to $\delta = 6h$.  To implement the Helmholtz filter, we approximate the Laplacian by central differences,
\begin{equation*}
  \Delta u \approx \Delta^h u = \frac{u(x - h) - 2u(x) + u(x + h)}{h^2},
\end{equation*}
with homogeneous Dirichlet boundary conditions, and define the discrete operator $A^h$, the inverse of $G^h$, by
\begin{equation*}
  A^h \bar{u} = -\delta^2\Delta^h\bar{u} + \bar{u} \quad \text{and} \quad A^hG^h = I.
\end{equation*}
The matrix $G^h=(A^h)^{-1}$ is symmetric positive definite with spectrum in $(0,1]$, so the results of Sections~\ref{sec:mitlar} and~\ref{sec:descent} apply to it directly.  The simulated data are obtained by filtering the true solution and adding $1\%$ random noise, $\bar{u} = G^hu - \epsilon$.  To generate $\epsilon$, we draw a vector with MATLAB's \texttt{randn}, normalize it to unit $L^2$ norm, and scale it so that $\| \epsilon \| = 0.01\, \| G^hu \|$.  $L^2$ norms are computed with the composite trapezoidal rule or the composite Simpson rule.  We use the regularization parameter $\alpha = 0.1$.  Since $\alpha\le\frac12$, the threshold in Theorem~\ref{5.3:them1} is $\tau_\alpha=\epsilon_0/\alpha$.  The stopping index is determined as follows:
\begin{enumerate}[Step 1.]
  \item At step $j$, compute the update norm $\norm{u_{j+1}-u_j}$.
  \item Compare the regularization parameter $\alpha$ with $\epsilon_0/\norm{u_{j+1}-u_j}$.
  \item Following \eqref{alphabound}, continue to the next step if $\alpha \ge \epsilon_0/\norm{u_{j+1}-u_j}$, and stop if $\alpha < \epsilon_0/\norm{u_{j+1}-u_j}$.
\end{enumerate}
For this demonstration we did not actually stop the iteration.  Instead we kept iterating and recorded the index at which the rule would have stopped.

Figure~\ref{fig:optimalJ} shows the noisy energy functional $E_{\epsilon}(u_j)$ along the Mitlar iterates for the noisy data $Gu = \bar{u} + \epsilon$.  The stopping rule of Theorem~\ref{5.3:them1} selects $J = 4$, marked by a green dot.  In this example the rule stops exactly where the functional attains its minimum and starts to increase, so the iteration is halted before it drifts toward the unstable naive inverse.  This is better than the theory guarantees.  Because \eqref{alphabound} is only sufficient and uses the noise bound $\epsilon_0$ rather than the noise itself, Theorem~\ref{5.3:them1} only ensures that the rule never stops later than the first increase of $E_\epsilon$; in general it may stop earlier.
\begin{figure}
  \centering
  \includegraphics[width=0.75\textwidth]{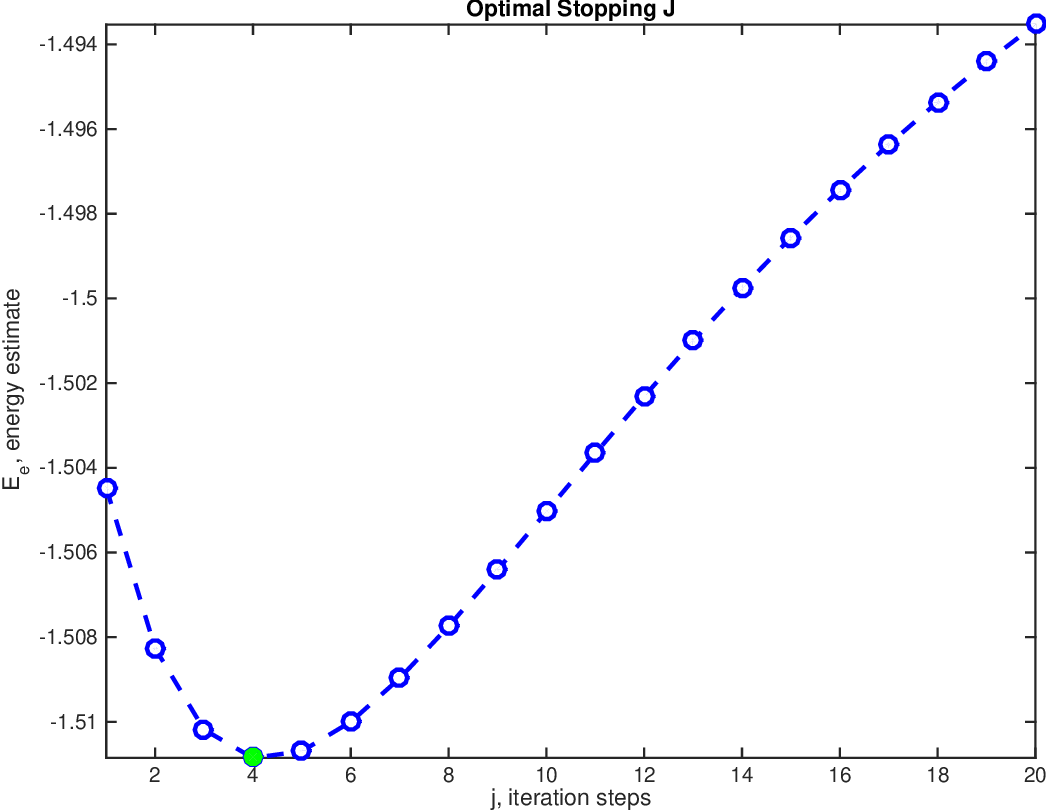}                  
  \caption{Noisy energy functional $E_\epsilon(u_J)$ for $J$ between $1$ and $20$.  The stopping criterion stops the iteration after $4$ steps (green dot with white face).}
\label{fig:optimalJ}
\end{figure}
\subsection{Comparison of four deconvolution algorithms}\label{numericalcomparison}
We assess the accuracy of Algorithm~\ref{alg:mitlar} by comparing, for each value of $\alpha$, its relative error with those of Tikhonov--Lavrentiev, iterated Tikhonov--Lavrentiev, and modified Tikhonov--Lavrentiev regularization with the same $\alpha$.  The codes are written in MATLAB (version R2014b).  The true solution is $u = \sin(\pi x) + 0.1\sin(100\pi x)$, sampled at $1001$ points on $[0, 2]$ (Figure~\ref{fig:comparetrue}), so the step size is $h = \frac{2}{1000}$.  The filter radius is $\delta = 0.01$.  We vary $\alpha$ from $1$ to $10^{-3}$ and compute $J=1$, $2$, and $3$ update steps for the two iterative methods.  For an approximation $u_{\rm appr}$ of $u$, the relative error is $\mathrm{Err}_{\rm rel} = \norm{u - u_{\rm appr}}/\norm{u}$.  The results are shown in Figures~\ref{fig:mitlarerror1}, \ref{fig:mitlarerror2}, and~\ref{fig:mitlarerror3}.  For every $\alpha$, Mitlar has the smallest error, and its accuracy improves as $J$ increases.  For larger $J$, the gap between Mitlar and the other three methods widens as $\alpha$ decreases.  For noise-free data, Theorem~\ref{thm:domination} and the remark following it predict this ordering for every $\alpha$ and $J$.  The widening gap reflects the factor $\delta^{2(J+1)}$ in \eqref{eqn:continuouserrorbound}, which iterated Tikhonov--Lavrentiev regularization lacks.
\begin{figure}
  \centering
    \includegraphics[width=0.75\textwidth]{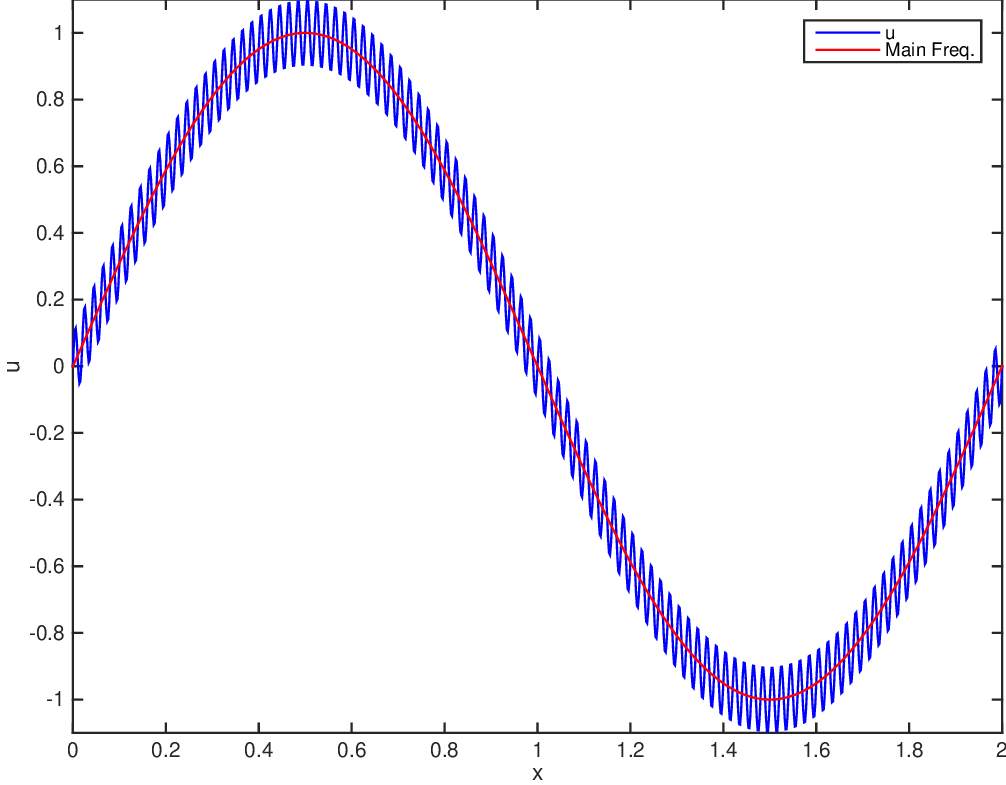}                  
    \caption{$u = \sin(\pi x) + 0.1\sin(100\pi x)$, used to compare Mitlar with Tikhonov--Lavrentiev, iterated Tikhonov--Lavrentiev, and modified Tikhonov--Lavrentiev regularization.} 
    \label{fig:comparetrue} 
\end{figure}
\begin{figure}
  \centering
    \includegraphics[width=0.75\textwidth]{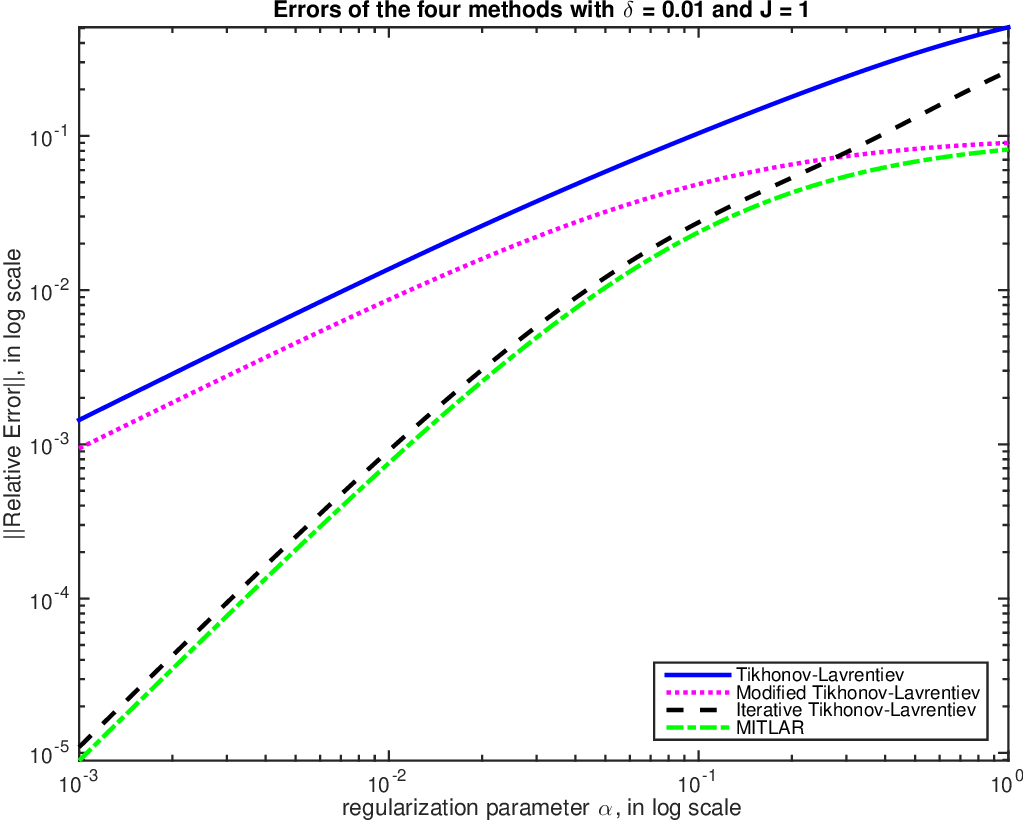}   
    \caption{Relative errors for the four algorithms over $\alpha = 10^{-3}$ to $\alpha = 1$ and $J = 1$.  Mitlar (modified iterated Tikhonov--Lavrentiev) has the lowest error over the entire range of regularization parameters. The gap between Mitlar and iterated Tikhonov--Lavrentiev is small across the whole $\alpha$ range.}
    \label{fig:mitlarerror1}
\end{figure}
\begin{figure}
  \centering
    \includegraphics[width=0.75\textwidth]{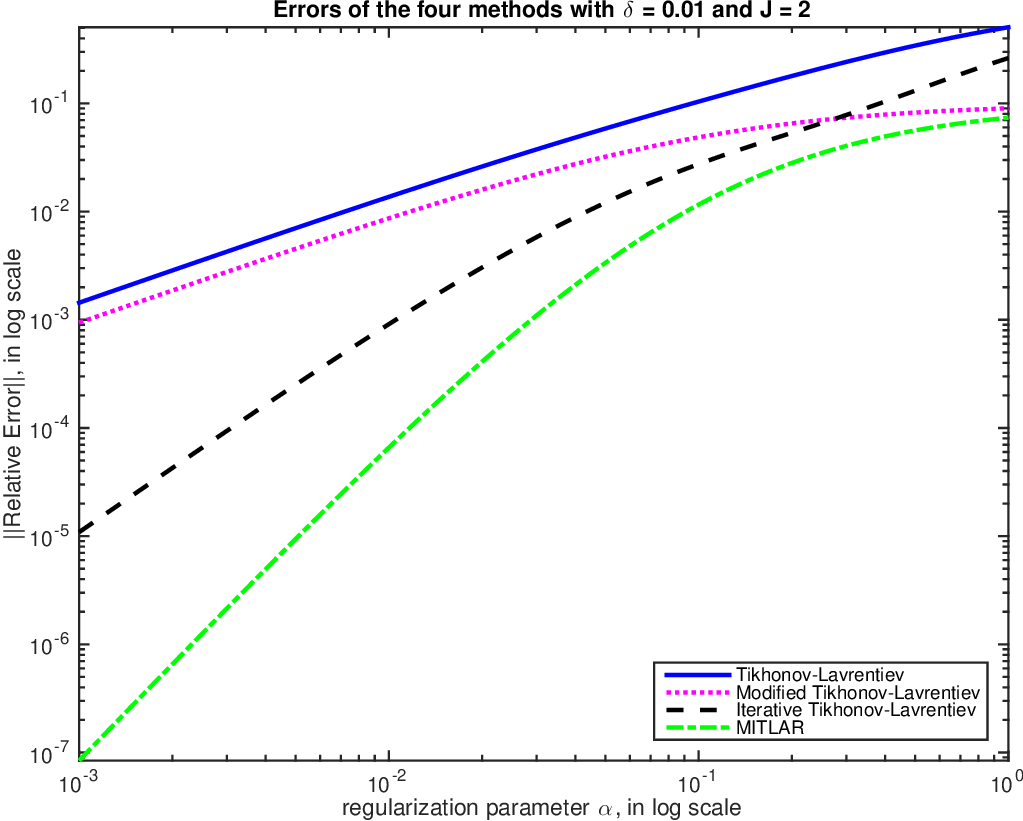}   
    \caption{Relative errors for the four algorithms over $\alpha = 10^{-3}$ to $\alpha = 1$ and $J = 2$.  Mitlar (modified iterated Tikhonov--Lavrentiev) has the lowest error over the entire range of regularization parameters. The gap between Mitlar and iterated Tikhonov--Lavrentiev increases as $\alpha \rightarrow 0$.}
    \label{fig:mitlarerror2}
\end{figure}
\begin{figure}
  \centering
    \includegraphics[width=0.75\textwidth]{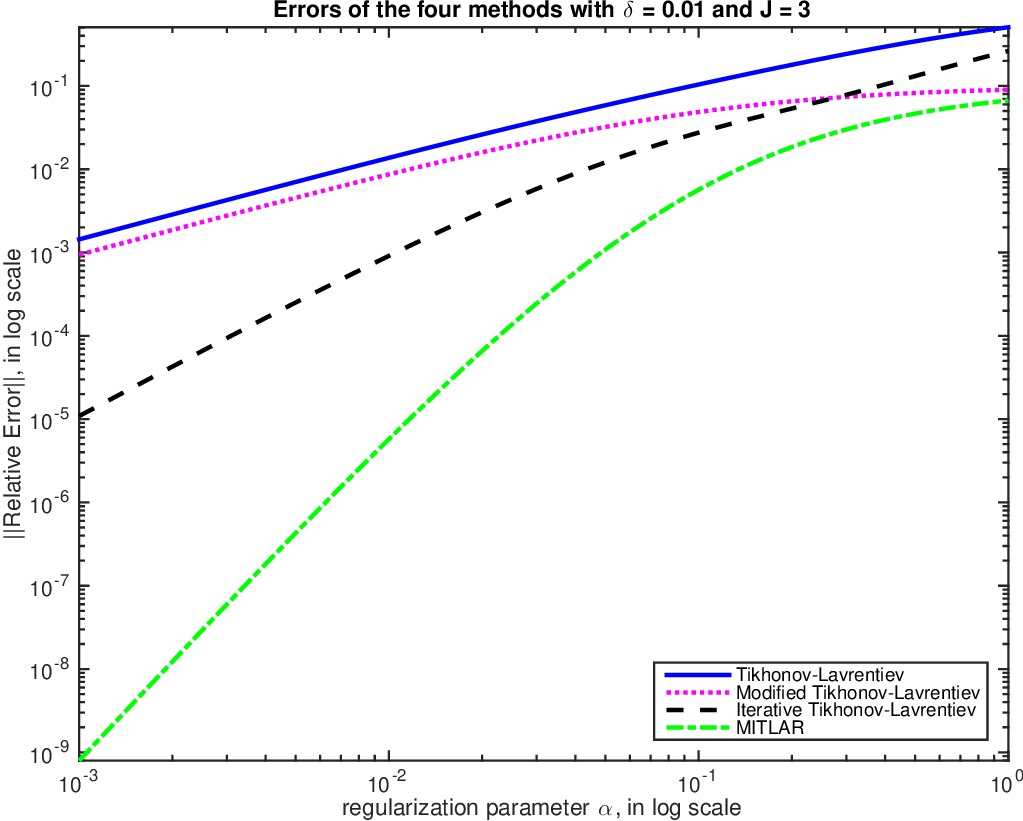}   
    \caption{Relative errors for the four algorithms over $\alpha = 10^{-3}$ to $\alpha = 1$ and $J = 3$.  Mitlar (modified iterated Tikhonov--Lavrentiev) has the lowest error over the entire range of regularization parameters.  The gap between Mitlar and iterated Tikhonov--Lavrentiev increases further as $\alpha \rightarrow 0$.}
    \label{fig:mitlarerror3}
\end{figure}
\subsection{Verification of convergence rates}\label{numericalconvergence}
We compute convergence rates for the Tikhonov--Lavrentiev family and compare them with the rates suggested by the regularization error $(\alpha\delta^2)^{J+1}$ (Theorem~\ref{thm:noisefree}).  Theorem~\ref{thm:discerror} and Corollary~\ref{cor:P1} bound the $L^2$ error; the $H^1$ rates reported here are an empirical complement to them.  The true solution on the domain $[0, 2]\times [0, 2]$ is
\begin{equation*}
u = \sin(\pi x) \sin(\pi y) + \sin(20\pi x)\sin(20\pi y),
\end{equation*}
shown in Figure~\ref{fig:convergetrue}.
\begin{figure}
  \centering
  \includegraphics[width=0.7\textwidth]{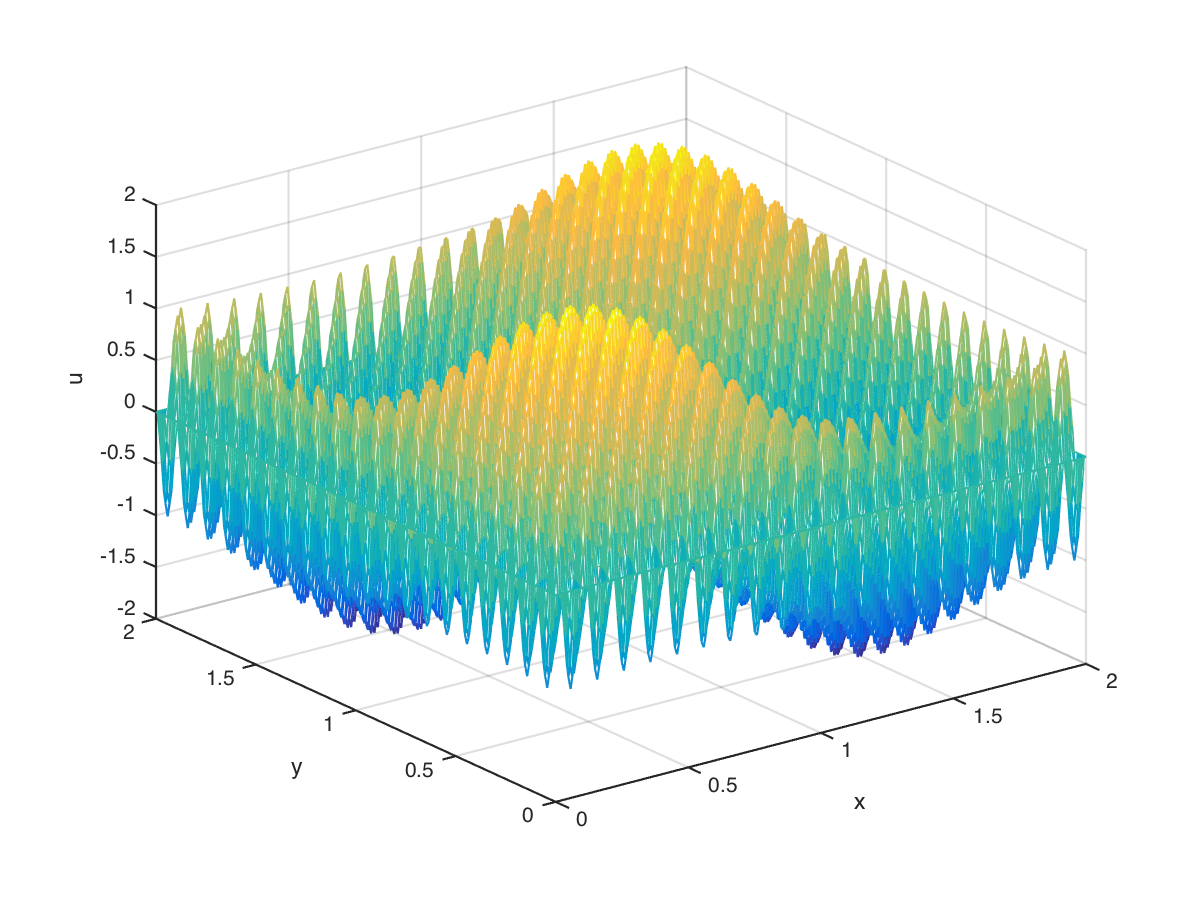}                  
  \caption{$u = \sin(\pi x) \sin(\pi y) + \sin(20\pi x)\sin(20\pi y)$, used in the convergence study.}
\label{fig:convergetrue}
\end{figure}
We discretize with the \texttt{square} mesh command of FreeFem++ \cite{MR3043640}, using $n$ intervals in each coordinate direction and continuous piecewise linear elements.  The filter radius is $\delta = 0.1(\frac{2\pi}{n})^{1/4} = \mathcal{O}(h^{1/4})$ and the regularization parameter is $\alpha = 0.1(\frac{2\pi}{n})^{1/2} = \mathcal{O}(h^{1/2})$, so that $\alpha\delta^2=\mathcal{O}(h)$.  Tables~\ref{table:MITLARconvrateJ=0}--\ref{table:ITLARconvrateJ=1} report the errors and observed convergence rates in the $H^1$ norm.
\begin{table}[t]
\centering
\caption{Convergence rates for modified Tikhonov--Lavrentiev regularization (Mitlar with $J = 0$).  The $H^1$ rates approach $1$, consistent with $\alpha\delta^2 = \mathcal{O}(h)$.}\label{table:MITLARconvrateJ=0}
\begin{tabular}{|l|l|l|}
\hline
$n$ & $H^1$ error & rate \\
\hline
60 & 45.2992 &  \\
120 & 30.0506 & 0.5921 \\
240 & 15.9409 & 0.9147 \\
480 & 8.43094 & 0.9190 \\
960 & 4.38127 & 0.9443 \\
\hline
\end{tabular}
\end{table}
\begin{table}[!tb]
\centering
\caption{Convergence rates for Mitlar with $J = 1$.  The $H^1$ rates approach $2$, consistent with $(\alpha\delta^2)^2 = \mathcal{O}(h^2)$.}\label{table:MITLARconvrateJ=1}
\begin{tabular}{|l|l|l|}
\hline
$n$ & $H^1$ error & rate \\
\hline
60 & 28.7953 &  \\
120 & 10.9148 & 1.3995 \\
240 & 2.90149 & 1.9114 \\
480 & 0.802423 & 1.8544 \\
960 & 0.216175 & 1.8922 \\
\hline
\end{tabular}
\end{table}
\begin{table}[!tb]
\centering
\caption{Convergence rates for Tikhonov--Lavrentiev regularization.  The expected rate is $1/2$, since $\alpha = \mathcal{O}(h^{1/2})$; the observed $H^1$ rates are higher (superconvergence).}\label{table:ITLARconvrateJ=0}
\begin{tabular}{|l|l|l|}
\hline
$n$ & $H^1$ error & rate \\
\hline
60 & 45.6474 &  \\
120 & 30.8419 & 0.5656 \\
240 & 16.8802 & 0.8696 \\
480 & 9.25259 & 0.8574 \\
960 & 5.02571 & 0.8805 \\
\hline
\end{tabular}
\end{table}
\begin{table}
\centering
\caption{Convergence rates for iterated Tikhonov--Lavrentiev regularization with $J = 1$.  The expected rate is $1$, since $\alpha^2 = \mathcal{O}(h)$; the observed $H^1$ rates are higher (superconvergence).}\label{table:ITLARconvrateJ=1}
\begin{tabular}{|l|l|l|}
\hline
$n$ & $H^1$ error & rate \\
\hline
60 & 29.1851 &  \\
120 & 11.4673 & 1.3477 \\
240 & 3.25182 & 1.8182 \\
480 & 0.966347 & 1.7506 \\
960 & 0.284427 & 1.7645 \\
\hline
\end{tabular}
\end{table}
\section{Conclusion}\label{sec:conclusion}
We analyzed Mitlar, an iterated and modified Lavrentiev regularization for deconvolving the Helmholtz-type differential filter.  The modification costs nothing extra, yet it changes the structure of the method.  The regularized operator is the same filter with the smaller radius $\sqrt\alpha\,\delta$, and $J$ updates perform van Cittert deconvolution of order $J$ of that filter.  As a result, the noise-free error is governed by $\alpha\delta^2$ rather than by $\alpha$.  It is of order $(\alpha\delta^2)^{s}$ for $u$ in the Hilbert scale of the Dirichlet Laplacian, $0\le s\le J+1$, and this rate is exact at $s=J+1$.  Tikhonov--Lavrentiev and iterated Tikhonov--Lavrentiev regularization instead have errors of order $\alpha$ and $\alpha^{J+1}$, with a lower bound $(\alpha/(1+\alpha))^{J+1}\|u\|$ that does not improve as $\delta\to0$.  For noise-free data, Mitlar is never less accurate than iterated Tikhonov--Lavrentiev regularization, for any filter with spectrum in $[0,1]$, any $\alpha$ and any $J$.  For noisy data, the noise amplification constant $(J+1)/\alpha$ is sharp.  A priori choices of either $\alpha$ or $J$ give explicit rates; when $J$ is the regularization parameter, there is no saturation.  The finite element analysis yields error constants that grow only linearly in $J$.

We also gave a computable stopping rule for the iteration.  The Mitlar iterates decrease the natural energy functional for every $\alpha>0$.  As long as the update norm exceeds $\epsilon_0/\min(\alpha,\frac12)$, each step reduces the weighted error $\|G^{1/2}(u-u_j)\|$.  The update norms decrease monotonically to zero, so the rule always terminates, and it stops no later than the first increase of the noisy energy.  Iterating until convergence, by contrast, produces the unstable naive inverse $G^{-1}\bar u$.  More information about the noise, such as a projection that annihilates it, could sharpen the rule.

Finally, the analysis of Sections~\ref{sec:mitlar} and~\ref{sec:descent} uses only the spectral calculus of $G$.  Proposition~\ref{prop:structure} is specific to the Helmholtz filter.  However, the identity $D_\alpha G=\phi(\alpha,G)$ with $\phi(\alpha,x)=x/((1-\alpha)x+\alpha)$ holds for any self-adjoint filter.  The same approach therefore applies to other filters whose transfer function behaves like that of the Pao filter at large scales.
\bibliographystyle{siam}
\bibliography{mitlar}

@article{AN01,
  author  = {A. K. Alekseev and I. M. Navon},
  title   = {The analysis of an ill-posed problem using multi-scale resolution and second-order adjoint techniques},
  journal = {Comput. Methods Appl. Mech. Engrg.},
  volume  = {190},
  number  = {15--17},
  year    = {2001},
  pages   = {1937--1953},
  doi     = {10.1016/S0045-7825(00)00299-1}
}

@book{BK04,
  author    = {A. B. Bakushinsky and M. Yu. Kokurin},
  title     = {Iterative Methods for Approximate Solution of Inverse Problems},
  series    = {Mathematics and Its Applications},
  volume    = {577},
  publisher = {Springer},
  address   = {Dordrecht},
  year      = {2004},
  doi       = {10.1007/978-1-4020-3122-9}
}

@book{BIL06,
  author    = {L. C. Berselli and T. Iliescu and W. J. Layton},
  title     = {Mathematics of Large Eddy Simulation of Turbulent Flows},
  series    = {Scientific Computation},
  publisher = {Springer},
  address   = {Berlin},
  year      = {2006}
}

@book{BB98,
  author    = {M. Bertero and P. Boccacci},
  title     = {Introduction to Inverse Problems in Imaging},
  publisher = {Institute of Physics Publishing},
  address   = {Bristol},
  year      = {1998}
}

@book{BS94,
  author    = {S. C. Brenner and L. R. Scott},
  title     = {The Mathematical Theory of Finite Element Methods},
  publisher = {Springer-Verlag},
  address   = {New York},
  year      = {1994}
}

@article{DSH01,
  author  = {N. G. Deen and T. Solberg and B. H. Hjertager},
  title   = {Large eddy simulation of the gas-liquid flow in a square cross-sectioned bubble column},
  journal = {Chem. Eng. Sci.},
  volume  = {56},
  year    = {2001},
  pages   = {6341--6349}
}

@article{DE06,
  author  = {A. Dunca and Y. Epshteyn},
  title   = {On the {S}tolz--{A}dams deconvolution model for the large-eddy simulation of turbulent flows},
  journal = {SIAM J. Math. Anal.},
  volume  = {37},
  number  = {6},
  year    = {2006},
  pages   = {1890--1902},
  doi     = {10.1137/S0036141003436302}
}

@article{EM09,
  author  = {H. W. Engl and C. Flamm and P. K{\"u}gler and J. Lu and S. M{\"u}ller and P. Schuster},
  title   = {Inverse problems in systems biology},
  journal = {Inverse Problems},
  volume  = {25},
  number  = {12},
  year    = {2009},
  pages   = {123014},
  doi     = {10.1088/0266-5611/25/12/123014}
}

@book{EHN96,
  author    = {H. W. Engl and M. Hanke and A. Neubauer},
  title     = {Regularization of Inverse Problems},
  series    = {Mathematics and Its Applications},
  volume    = {375},
  publisher = {Kluwer},
  address   = {Dordrecht},
  year      = {1996}
}

@article{G86,
  author  = {M. Germano},
  title   = {Differential filters for the large eddy numerical simulation of turbulent flows},
  journal = {Phys. Fluids},
  volume  = {29},
  year    = {1986},
  pages   = {1755--1757}
}

@article{Geu97,
  author  = {B. J. Geurts},
  title   = {Inverse modeling for large-eddy simulation},
  journal = {Phys. Fluids},
  volume  = {9},
  year    = {1997},
  pages   = {3585--3587}
}

@article{GCS10,
  author  = {M. Ghizaru and P. Charbonneau and P. K. Smolarkiewicz},
  title   = {Magnetic cycles in global large-eddy simulations of solar convection},
  journal = {Astrophys. J. Lett.},
  volume  = {715},
  year    = {2010},
  pages   = {L133--L137}
}

@book{Gri85,
  author    = {P. Grisvard},
  title     = {Elliptic Problems in Nonsmooth Domains},
  publisher = {Pitman},
  address   = {Boston},
  year      = {1985}
}

@article{GOP04,
  author  = {J.-L. Guermond and J. T. Oden and S. Prudhomme},
  title   = {Mathematical perspectives on large eddy simulation models for turbulent flows},
  journal = {J. Math. Fluid Mech.},
  volume  = {6},
  year    = {2004},
  pages   = {194--248}
}

@article{H94,
  author  = {P. C. Hansen},
  title   = {Regularization {T}ools: {A} {M}atlab package for analysis and solution of discrete ill-posed problems},
  journal = {Numer. Algorithms},
  volume  = {6},
  year    = {1994},
  pages   = {1--35}
}

@article{MR3043640,
  author  = {F. Hecht},
  title   = {New development in {FreeFem++}},
  journal = {J. Numer. Math.},
  volume  = {20},
  number  = {3--4},
  year    = {2012},
  pages   = {251--265}
}

@article{IMMS12,
  author  = {R. Ingram and C. C. Manica and N. Mays and I. Stanculescu},
  title   = {Convergence analysis of a fully discrete family of iterated deconvolution methods for turbulence modeling with time relaxation},
  journal = {Adv. Numer. Anal.},
  volume  = {2012},
  year    = {2012},
  pages   = {Art. ID 162539, 32 pp.},
  doi     = {10.1155/2012/162539}
}

@article{KC79,
  author  = {J. T. King and D. Chillingworth},
  title   = {Approximation of generalized inverses by iterated regularization},
  journal = {Numer. Funct. Anal. Optim.},
  volume  = {1},
  number  = {5},
  year    = {1979},
  pages   = {499--513},
  doi     = {10.1080/01630567908816031}
}

@article{LT10,
  author  = {A. Labovsky and C. Trenchea},
  title   = {Approximate deconvolution models for magnetohydrodynamics},
  journal = {Numer. Funct. Anal. Optim.},
  volume  = {31},
  number  = {12},
  year    = {2010},
  pages   = {1362--1385}
}

@book{L67,
  author    = {M. M. Lavrentiev},
  title     = {Some Improperly Posed Problems of Mathematical Physics},
  publisher = {Springer},
  address   = {New York},
  year      = {1967}
}

@article{LL05,
  author  = {W. J. Layton and R. Lewandowski},
  title   = {Residual stress of approximate deconvolution large eddy simulation models of turbulence},
  journal = {J. Turbul.},
  volume  = {7},
  year    = {2006},
  pages   = {Paper N46, 1--21}
}

@book{LR12,
  author    = {W. J. Layton and L. G. Rebholz},
  title     = {Approximate Deconvolution Models of Turbulence: Analysis, Phenomenology and Numerical Analysis},
  series    = {Lecture Notes in Mathematics},
  volume    = {2042},
  publisher = {Springer},
  address   = {Heidelberg},
  year      = {2012}
}

@incollection{LN96,
  author    = {F. Liu and M. Z. Nashed},
  title     = {Convergence of regularized solutions of nonlinear ill-posed problems with monotone operators},
  booktitle = {Partial Differential Equations and Applications},
  editor    = {P. Marcellini and G. Talenti and E. Vesentini},
  series    = {Lecture Notes in Pure and Appl. Math.},
  volume    = {177},
  publisher = {Dekker},
  address   = {New York},
  year      = {1996},
  pages     = {353--361}
}

@article{MK12,
  author  = {C. C. Manica and S. {Kaya Merdan}},
  title   = {Convergence analysis of the finite element method for a fundamental model in turbulence},
  journal = {Math. Models Methods Appl. Sci.},
  volume  = {22},
  number  = {11},
  year    = {2012},
  pages   = {1250033},
  doi     = {10.1142/S0218202512500339}
}

@article{MS09,
  author  = {C. C. Manica and I. Stanculescu},
  title   = {Numerical analysis of {L}eray--{T}ikhonov deconvolution models of fluid motion},
  journal = {Comput. Math. Appl.},
  volume  = {60},
  year    = {2010},
  pages   = {1440--1456}
}

@phdthesis{Mays11,
  author = {N. Mays},
  title  = {Iterated Regularization Methods for Solving Inverse Problems},
  school = {University of Pittsburgh},
  year   = {2011}
}

@article{MM97,
  author  = {R. Mittal and P. Moin},
  title   = {Suitability of upwind-biased finite difference schemes for large-eddy simulation of turbulent flows},
  journal = {AIAA J.},
  volume  = {35},
  year    = {1997},
  pages   = {1415--1417}
}

@article{M84,
  author  = {C.-H. Moeng},
  title   = {A large-eddy-simulation model for the study of planetary boundary-layer turbulence},
  journal = {J. Atmos. Sci.},
  volume  = {41},
  year    = {1984},
  pages   = {2052--2062}
}

@article{NY98,
  author  = {K. Nadaoka and H. Yagi},
  title   = {Shallow-water turbulence modeling and horizontal large-eddy computation of river flow},
  journal = {J. Hydraul. Eng.},
  volume  = {124},
  year    = {1998},
  pages   = {493--500}
}

@article{N84,
  author  = {F. Natterer},
  title   = {Error bounds for {T}ikhonov regularization in {H}ilbert scales},
  journal = {Applicable Anal.},
  volume  = {18},
  number  = {1--2},
  year    = {1984},
  pages   = {29--37},
  doi     = {10.1080/00036818408839508}
}

@incollection{PSB01,
  author    = {U. Piomelli and A. Scotti and E. Balaras},
  title     = {Large-eddy simulations of turbulent flows, from desktop to supercomputer},
  booktitle = {Vector and Parallel Processing -- VECPAR 2000},
  editor    = {J. M. L. M. Palma and J. Dongarra and V. Hern{\'a}ndez},
  series    = {Lecture Notes in Comput. Sci.},
  volume    = {1981},
  publisher = {Springer},
  address   = {Berlin, Heidelberg},
  year      = {2001},
  pages     = {551--578},
  doi       = {10.1007/3-540-44942-6_45}
}

@book{S07,
  author    = {T. Schuster},
  title     = {The Method of Approximate Inverse: Theory and Applications},
  series    = {Lecture Notes in Mathematics},
  volume    = {1906},
  publisher = {Springer},
  address   = {Berlin},
  year      = {2007}
}

@article{SF97,
  author  = {K. B. Shah and J. H. Ferziger},
  title   = {A fluid mechanicians view of wind engineering: {L}arge eddy simulation of flow past a cubic obstacle},
  journal = {J. Wind Eng. Ind. Aerodyn.},
  volume  = {67--68},
  year    = {1997},
  pages   = {211--224}
}

@article{Sma63,
  author  = {J. Smagorinsky},
  title   = {General circulation experiments with the primitive equations. {I}. {T}he basic experiment},
  journal = {Mon. Weather Rev.},
  volume  = {91},
  number  = {3},
  year    = {1963},
  pages   = {99--164}
}

@article{SA99,
  author  = {S. Stolz and N. A. Adams},
  title   = {An approximate deconvolution procedure for large-eddy simulation},
  journal = {Phys. Fluids},
  volume  = {11},
  number  = {7},
  year    = {1999},
  pages   = {1699--1701}
}

@article{SAK01,
  author  = {S. Stolz and N. A. Adams and L. Kleiser},
  title   = {The approximate deconvolution model for large-eddy simulations of compressible flows and its application to shock-turbulent-boundary-layer interaction},
  journal = {Phys. Fluids},
  volume  = {13},
  year    = {2001},
  pages   = {2985--3001}
}

@book{TA77,
  author    = {A. N. Tikhonov and V. Y. Arsenin},
  title     = {Solutions of Ill-Posed Problems},
  publisher = {V. H. Winston \& Sons; Halsted Press},
  address   = {Washington, DC; New York},
  year      = {1977}
}

@book{TA79,
  author    = {A. N. Tikhonov and V. Y. Arsenin},
  title     = {Methods for Solving Ill-Posed Problems},
  edition   = {2nd},
  publisher = {Nauka},
  address   = {Moscow},
  year      = {1979},
  note      = {In Russian}
}

@book{V82,
  author    = {G. M. Vainikko},
  title     = {Methods for Solving Linear Ill-Posed Problems in {H}ilbert Spaces},
  publisher = {Tartu State University},
  address   = {Tartu},
  year      = {1982},
  note      = {In Russian}
}

@book{VV86,
  author    = {G. M. Vainikko and A. Yu. Veretennikov},
  title     = {Iteration Procedures in Ill-Posed Problems},
  publisher = {Nauka},
  address   = {Moscow},
  year      = {1986},
  note      = {In Russian}
}

@book{V02,
  author    = {C. R. Vogel},
  title     = {Computational Methods for Inverse Problems},
  series    = {Frontiers in Applied Mathematics},
  volume    = {23},
  publisher = {SIAM},
  address   = {Philadelphia},
  year      = {2002}
}
\end{document}